\documentclass{article}

\usepackage{amsmath,amssymb,amsthm}
\usepackage{hyperref}

\usepackage{tikz}

\usepackage{graphicx,epsfig,xcolor,stmaryrd}
\usepackage{float}

\def\R{\mathbb R}

\def\Z{\mathbb Z}
\def\N{\mathbb N}

\numberwithin{equation}{section}

\newtheorem{theorem}{Theorem}

\newtheorem{lemma}[theorem]{Lemma}
\newtheorem{proposition}[theorem]{Proposition}

\numberwithin{theorem}{section}

\makeindex

\begin{document}

\title{Not all sub-Riemannian minimizing geodesics are smooth}


\author{
Y. Chitour\thanks{Universit\'e Paris-Saclay, Centralesupelec, CNRS, Laboratoire des signaux et syst\`emes, UMR CNRS 8506, 91190 Gif-sur-Yvette, France (\texttt{yacine.chitour@centralesupelec.fr})}
\and F. Jean\thanks{Unit\'{e} de Math\'{e}matiques Appliqu\'{e}es, ENSTA, Institut Polytechnique de Paris, 91120 Palaiseau, France (\texttt{frederic.jean@ensta.fr})} \and  R. Monti
\and L. Rifford\thanks{Universit\'e C\^ote d'Azur, CNRS, Labo.\ J.-A.\ Dieudonn\'e,  UMR CNRS 7351, Parc Valrose, 06108 Nice Cedex 02, France \& AIMS Senegal, Km 2, Route de Joal, Mbour, Senegal ({\tt ludovic.rifford@math.cnrs.fr})} \and L. Sacchelli\thanks{Universit\'{e} C\^{o}te d’Azur, Inria, CNRS, LJAD, France ({\tt ludovic.sacchelli@inria.fr})}
\and M. Sigalotti\thanks{Laboratoire Jacques-Louis Lions, Sorbonne Universit\'{e}, Universit\'{e} de Paris, CNRS, Inria, Paris, France (\texttt{mario.sigalotti@inria.fr})}
\and A. Socionovo\thanks{Laboratoire Jacques-Louis Lions, Sorbonne Universit\'{e}, Universit\'{e} de Paris, CNRS, Inria, Paris, France (\texttt{alessandro.socionovo@sorbonne-universite.fr})}}

\date{\today}


\maketitle

\begin{abstract}
A longstanding open question in sub-Riemannian geometry is the following: are sub-Riemannian length minimizers smooth? We give a negative answer to this question, exhibiting an example of a $C^2$ but not $C^3$
length-minimizer of a real-analytic (even polynomial) sub-Riemannian structure.
\end{abstract}


\section{Introduction}

Let $M$ be a smooth, connected manifold of dimension $n\geq 3$, equipped with a sub-Riemannian structure $(\Delta,g)$. This structure consists of a  bracket generating distribution $\Delta$ of rank $m\leq n$ on $M$, that is, a smooth subbundle of $TM$ of dimension $m$ generated locally by $m$ smooth vector fields $X^1, \ldots,X^m$ satisfying the H\"ormander condition
\begin{eqnarray*}
\mbox{Lie} \left\{ X^1, \ldots,X^m \right\} (x) = T_xM \qquad \forall\,x \in M,
\end{eqnarray*}
and a smooth metric $g$ on $\Delta$. By the Chow-Rashevsky Theorem, $M$ is horizontally path-connected with respect to $\Delta$. In other words, for any two points $x,y \in M$, there exists a horizontal path
 connecting them, i.e., an absolutely continuous curve $\gamma : [0,T] \rightarrow M$ satisfying
\begin{eqnarray*}
\dot{\gamma}(t) \in \Delta (\gamma(t)) \qquad  \mbox{for almost every } t \in [0,T],\qquad \gamma(0)=x,\quad \gamma(T)=y.
\end{eqnarray*}
The sub-Riemannian distance $d_{SR}$ associated with  $(\Delta,g)$ is defined as the infimum of the lengths of horizontal paths connecting two points:  for every $x,y \in M$,
\[
d_{SR}(x,y) := \inf \left\{ \mbox{length}^g(\gamma) \, \vert \, \gamma : [0,T] \rightarrow M \mbox{ horizontal s.t. } \gamma(0)=x, \, \gamma(T)=y \right\},
\]
where the length of a horizontal path, computed using the norm $|\cdot|_g$ induced by the metric $g$, is given by
\[
 \mbox{length}^g(\gamma) :=\int_0^T \left|\dot{\gamma}(t)\right|_g \, dt.
 \]
Sub-Riemannian geometry explores the metric and geometric properties of the resulting metric space $(M,d_{SR})$. In the special case where $m=n$, the framework reduces to the Riemannian case, where all absolutely continuous curves are horizontal. A distinctive feature of sub-Riemannian geometry, when $m<n$, is the presence of singular horizontal paths. These paths are central to one of the most challenging problems in the field: understanding the regularity of the horizontal curves minimizing the sub-Riemannian distance $d_{SR}$.
The aim of this paper is to demonstrate that these curves are not necessarily smooth.

The Hopf-Rinow theorem remains valid in the sub-Riemannian setting. For further details on the notions and results of sub-Riemannian geometry mentioned in the introduction, we refer the reader to Bella\"iche's monograph \cite{bellaiche96}, or to the books by Montgomery \cite{montgomery02},  Agrachev, Barilari and Boscain \cite{abb20}, and the fourth author \cite{riffordbook}. If the metric space $(M,d_{SR})$ is complete, then for any points $x,y\in M$, there exists a horizontal path $\gamma : [0,T] \rightarrow M$ between $x$ and $y$ that minimizes length, {\it i.e.}, $d_{SR}(x,y)=\mbox{length}^g(\gamma)$. When reparametrized with constant speed, such a path is referred to as \emph{a minimizing geodesic}. It minimizes the energy $\int_0^T |\dot{\gamma}(t)|_g^2 \, dt$ between its endpoints in fixed time $T$. By the Pontryagin maximum principle, a minimizing geodesic is either the projection of a so-called normal extremal or strictly singular. In the former case, it is smooth because it is the projection of a solution of a smooth Hamiltonian system associated with $(\Delta,g)$ in $T^*M$. In the latter case, which cannot be ruled out due to a famous example by Montgomery \cite{montgomery94}, its regularity remains uncertain. So far, results on the regularity of strictly singular minimizing geodesics are known only in a few cases. Building on an earlier result by Leonardi and Monti \cite{lm08}, Hakavuori and Le Donne \cite{hl16} showed that minimizing geodesics cannot exhibit corner-type singularities. In \cite{bcjps20}, Barilari, Chitour, Jean, Prandi, and Sigalotti used this result to prove that minimizing geodesics for rank $2$ sub-Riemannian structures with step up to $4$ are of class $C^1$. Then Monti, Pigati and Vittone proved in \cite{mpv18} the everywhere existence of a tangent line in the tangent cone to minimizing geodesics, a step towards their $C^1$ regularity.
In the case where $M$ and $\Delta$ are real-analytic, Sussmann \cite{sussmann14} (see also \cite{bprPreprint}) showed that every strictly singular minimizing geodesic $\gamma : [0,T] \rightarrow M$ is smooth (and analytic whenever $g$ is analytic) on an open dense subset of $[0,T]$. In this context, Belotto da Silva, Figalli, Parusi\'nski and Rifford \cite{bfpr22} proved that minimizing geodesics for rank-$2$ sub-Riemannian structures in dimension $3$ are semianalytic. In addition, Le Donne, Paddeu, and Socionovo \cite{lps24} obtained the same result in any dimension under the assumption that the distribution has rank $2$, is equiregular, and is metabelian. In the last two cases, the above-mentioned result of Hakavuori and Le Donne allows to show that the minimizing geodesics are indeed of class $C^{1,\alpha}$ for some $\alpha\in (0,1]$.

In this paper we present  a counterexample to the smoothness of minimizing geodesics, which has two main motivations. On the one hand, it is motivated by the result of \cite{bfpr22}, which follows from a precise description of the orbits of the singular line field given by the trace of the distribution $\Delta$ on the Martinet surface $\Sigma_{\Delta}$. On the other hand, it also has its origins in \cite{lm08}, where it is explained how to construct examples of singular curves with any kind of singularity, and in the study of (non-)minimality of half-parabolic type curves in~\cite{monti}.

We consider the sub-Riemannian structure $(\Delta,g)$ in $\R^3$ with coordinates $(x_1,x_2,x_3)$, generated by an orthonormal family of vector fields $\{X^1,X^2\}$ defined as
\[
X^1 = \partial_1 \quad \mbox{and} \quad X^2 = \partial_2 + P(x)^2 \partial_3,
\]
where
\[
P(x) = x_1^{2} - x_2^m   \qquad \forall x =(x_1,x_2,x_3) \in \R^3,
\]
and $m$ is an odd integer satisfying $m\geq 5$.
Besides the motivations described above, the counterexample took this particular form after a study of several types of possible examples in~\cite{socionovo}, its structure (in particular with the square of $P$) being inspired by the Liu--Sussmann example~\cite{ls95}.

The Martinet surface of this distribution is given by
\[
\Sigma_{\Delta} := \left\{ x \in \R^3 \, \vert \, [X^1,X^2](x) \in \Delta(x) \right\} = \left\{ Q = 0  \right\} \quad \mbox{with} \quad Q(x)=  \partial_1 P^2(x) = 4x_1 \left(x_1^{2} -  x_2^m\right).
\]
As any horizontal curve contained in $\Sigma_{\Delta}$ is singular, the curve $\bar{\gamma} : [0,\infty) \rightarrow \Sigma_{\Delta}$ given by
\begin{eqnarray}\label{DEFgammabar}
\bar{\gamma}(t) = \left( t^{\bar{m}},t,0\right) \quad \forall t \geq 0, \quad \mbox{with} \quad \bar{m}:=\frac{m}{2},
\end{eqnarray}
is a singular horizontal path of the distribution $\Delta$. Moreover it is not smooth whenever $\bar{m}$ is not an integer. We show the following result.

\begin{theorem}\label{MainTHM}
For every odd integer $m\geq 5$ and for any sufficiently small $\epsilon>0$, the curve $\bar{\gamma}|_{ [0,\epsilon]}$ is the unique horizontal path minimizing the distance between $\bar{\gamma}(0)$ and $\bar{\gamma}(\epsilon)$ with respect to $(\Delta,g)$. Furthermore, its arc length reparametrization is of class $C^{\bar{m}-1/2}$ but not $C^{\bar{m}+1/2}$.
\end{theorem}

The example with the lowest regularity provided by Theorem \ref{MainTHM} is of class $C^2$ but not $C^3$, and it is achieved for $m=5$. In \cite{socionovo}, it is shown that this result is sharp in the sense that, for every $\epsilon >0$, the curve $\bar{\gamma}|_{ [0,\epsilon]}$ is not a minimizer as soon as it is $C^1$  but not $C^2$. Theorem \ref{MainTHM} disproves the claim that sub-Riemannian minimizing geodesics are always of class $C^{\infty}$ but leaves open the question of $C^1$ or $C^2$ regularity. The remainder of the paper is dedicated to the proof of Theorem \ref{MainTHM}. We outline the general structure of the proof in Section \ref{SECProof}, referring to subsequent sections for the required technical results.\medskip

Finally, the authors wish to emphasize the special contribution of one of them, A.~Socionovo.

\paragraph*{Research funding.} This project has received funding from the European
Union's Horizon 2020 research and innovation programme under the Marie
Sk{\l}odowska-Curie grant agreement No 101034255.
\includegraphics[width=0.65cm]{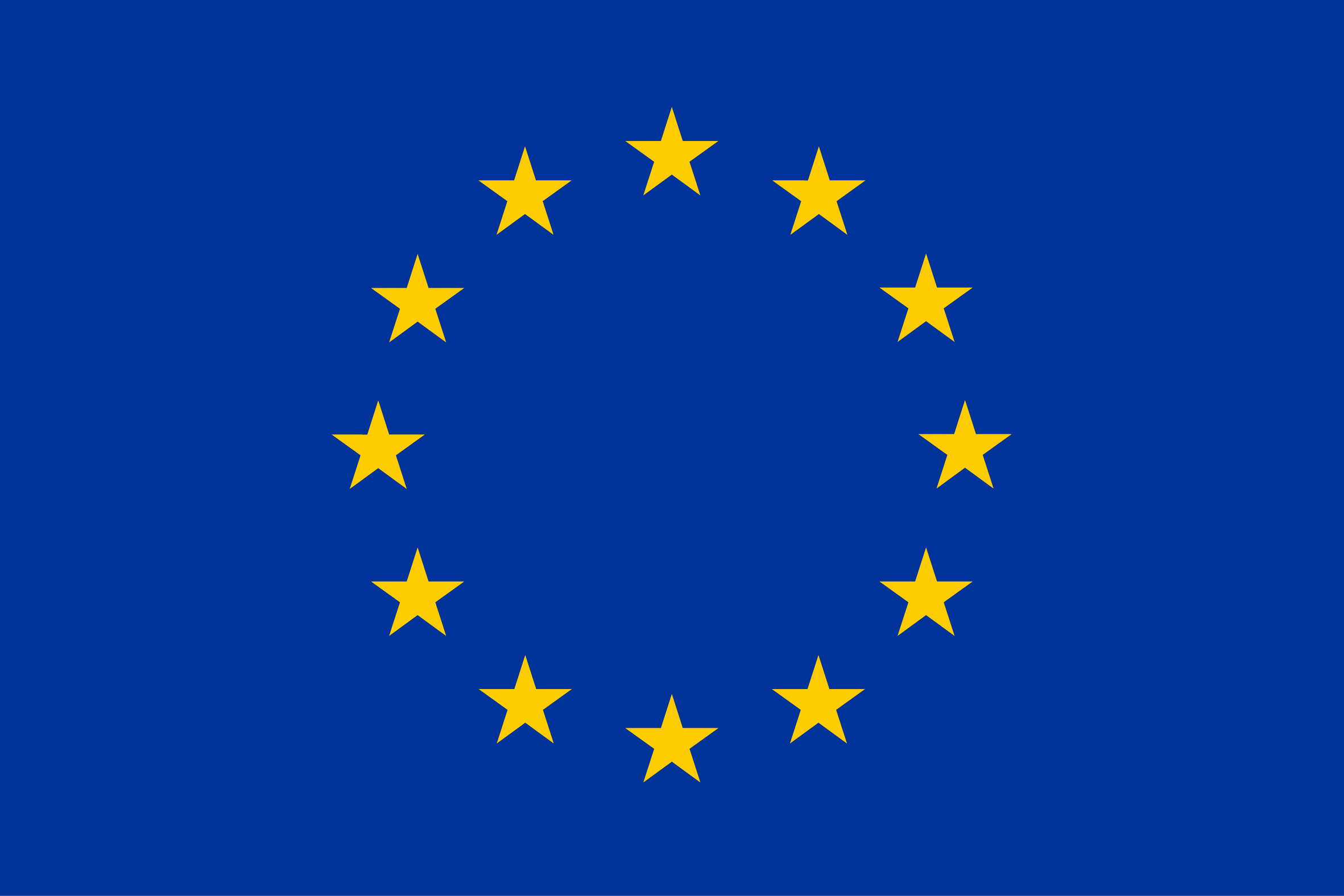}

\section{Proof of Theorem \ref{MainTHM}}\label{SECProof}

The purpose of this section is to outline the proof of Theorem \ref{MainTHM}, presenting its structure and referring to the subsequent sections for the detailed demonstration of the required results.

The proof of Theorem \ref{MainTHM} proceeds by contradiction. We fix $\epsilon>0$ and we assume that there exists a horizontal path
\[
\gamma_{\epsilon}\, : \, \left[0,\mbox{length}^g \left(\gamma_{\epsilon}\right)\right] \rightarrow \R^3,
\]
parametrized by arc length (with respect to $g$), which is minimizing from $\bar{\gamma}(0)$ to $\bar{\gamma}(\epsilon)$ but is not identical to $\bar{\gamma}_{\epsilon}:=\bar{\gamma}|_{ [0,\epsilon]}$, up to reparametrization. Then, we have
\begin{eqnarray}\label{Lengths}
 d_{SR}\left(\bar{\gamma}(0),\bar{\gamma}(\epsilon)\right) = \mbox{length}^g \left(\gamma_{\epsilon}\right)  \leq  \mbox{length}^g \left( \bar{\gamma}_{\epsilon}\right).
\end{eqnarray}
After a careful study of such a curve $\gamma_{\epsilon}$, we will derive a contradiction if $\epsilon>0$ is sufficiently small. The proof consists in several steps that we now describe. The first step is a straightforward consequence of the fact that $\gamma_{\epsilon}$ is necessarily a regular horizontal path.\\

\noindent {\bf Step 1:} Projection of the minimization problem to the plane $(x_1,x_2)$.\\
Being a minimizing horizontal path between two points $y, z \in \R^3$ with respect to $(\Delta,g)$ is equivalent to having a projection onto the $(x_1,x_2)$-plane that minimizes the Euclidean length among all curves joining $(y_1,y_2)$ to $(z_1,z_2)$ and along which the integral of $P^2dx_2$ is equal to $z_3 - y_3$. Thus, the projection of $\gamma_{\epsilon}$ onto the $(x_1,x_2)$-plane, denoted by $\omega_{\epsilon}$, has Euclidean length $L(\omega_{\epsilon})= \mbox{length}^g (\gamma_{\epsilon})$, is parametrized by (Euclidean) arc length, and minimizes the Euclidean length $L(\zeta)$ among all Lipschitz curves $\zeta:[0,\tau] \rightarrow \R^2$ satisfying the following conditions:
\begin{eqnarray}\label{OptimPlane}
\zeta(0)=A_0:=(0,0), \quad \zeta(\tau)= A_{\epsilon}:= (\epsilon^{\bar{m}},\epsilon), \quad \mbox{and} \quad \int_{\zeta} P(x)^2\, dx_2 = 0.
\end{eqnarray}
Furthermore, since $\gamma_{\epsilon}$ is not identical to $\bar{\gamma}_{\epsilon}$, it cannot be contained in $\Sigma_{\Delta}$ and then is not a singular curve. Thus it must correspond to the projection of a normal extremal. Consequently (see~\cite{abb20,lm08,riffordbook}), the curve $\omega_{\epsilon}:[0,L(\omega_{\epsilon})]\rightarrow \R^2$ is associated with a function $\theta_{\epsilon}: [0,L(\omega_{\epsilon})]\rightarrow \R$, where $\theta_{\epsilon}(0)\in (-\pi,\pi]$, and a constant $\lambda_{\epsilon} \in \R$, such that the following system holds:
\begin{eqnarray}\label{SYSomega}
\dot{\omega}_{\epsilon}(t) & = &  \left( \begin{matrix}
\cos \theta_{\epsilon}(t)\\
 \sin \theta_{\epsilon}(t)
 \end{matrix}
 \right)
 \quad \mbox{and} \quad
\dot{\theta}_{\epsilon}(t)  =  \lambda_{\epsilon} \, Q\left(\omega_{\epsilon}(t)\right)
\qquad \forall t \in [0,L(\omega_{\epsilon})].
\end{eqnarray}
In particular, the functions $\omega_{\epsilon}$ and $\theta_{\epsilon}$ are analytic. By construction, $L(\omega_{\epsilon}) = \mbox{length}^g (\gamma_{\epsilon})$ is no greater than $L(\bar{\omega}_{\epsilon})=\mbox{length}^g (\bar{\gamma}_{\epsilon})$, where $\bar{\omega}_{\epsilon}$ is the projection of $\bar{\gamma}_{\epsilon}$. \\

Before delving into the study of $\omega_{\epsilon}$, the next step is to address a problem of calculus of variations with constraints which will be instrumental in proving the main result of Step 3 and in reaching a contradiction in Step 4. \\


\noindent {\bf Step 2:} A problem of calculus of variations with constraints relying on $P$-sublevel sets.\\
The following result concerns the length of curves remaining in the region where $P \leq \rho$. 
The proof is provided in Appendix \ref{SECPROPsublevelsets}.

\begin{proposition}\label{PROPsublevelsets}
Given $\rho>0$, the following properties hold.
\begin{itemize}
  \item[(i)] For every $K>0$, there exist $C(K)>0$ and $\epsilon_0>0$ such that for all $\epsilon \in (0,\epsilon_0)$, if $\rho \in (0,K\epsilon^{3\bar{m}-1})$, then every Lipschitz curve $\zeta:[0,\tau]\rightarrow \R^2$ satisfying
\begin{eqnarray}\label{CalcVar}
\zeta(0)=A_0, \quad \zeta(\tau) = A_{\epsilon}, \quad \mbox{and} \quad \zeta_1(t) >0,  \  P(\zeta(t)) \leq \rho \quad \forall t \in [0,\tau],
\end{eqnarray}
admits the  following lower bound on its length,
\begin{eqnarray}\label{19dec1}
L\left( \zeta \right)  \geq L\left(\bar{\omega}_{\epsilon}\right) - C(K) \rho^{1-\frac{1}{m}}.
\end{eqnarray}

  \item[(ii)] Define the functions $f_{\rho}$ and $\Gamma_{\rho}$ on $[0,+\infty)$ by
\[
f_{\rho}(t):= \left( t^m+\rho\right)^{\frac{1}{2}} \quad \mbox{and} \quad \Gamma_{\rho}(t) := \left(f_{\rho}(t),t\right) \qquad \forall t \geq 0,
\]
and for any interval $I\subset [0,+\infty)$, set $\Gamma_{\rho} (I) := \left\{ \Gamma_{\rho}(t) \, \vert \, t \in I \right\}$. Then, for all $s\geq t\geq 0$, 
\begin{eqnarray}\label{PROPsublevelsetslast}
L\left(\Gamma_{\rho}([t,s])\right) -  L\left(\left[ \Gamma_{\rho}(t),\Gamma_{\rho}(s)\right])\right) \leq \frac{ \bar{m}^2}{2} \left(\bar{m} -1\right) s^{m-3} (s-t)^2.
\end{eqnarray}
\end{itemize}
\end{proposition}

The next step consists in conducting a detailed analysis of the curve $\omega_{\epsilon}$ to describe its shape as precisely as possible.  To simplify notation, we now omit the $\epsilon$ subscript and write $\omega$, $\bar{\omega}$, $\theta$, and $\lambda$ instead of $\omega_{\epsilon}$, $\bar{\omega}_{\epsilon}$, $\theta_{\epsilon}$, and $\lambda_{\epsilon}$, respectively.  \\

\noindent {\bf Step 3:} Anatomy of $\omega$.\\
As we shall show, the curve $\omega$ cannot be injective and must therefore admit at least one loop. We define a loop of $\omega$ as any curve $\ell$ corresponding to the restriction of $\omega$ to an interval $J_{\ell}=[s_{\ell}^{-},s_{\ell}^{+}] \subset [0,L(\omega)]$, where $s_{\ell}^{+}\neq s_{\ell}^{-}$, such that $\omega(s_{\ell}^{-})=\omega(s_{\ell}^{+})$. The proof of the following result occupies the entire Section \ref{SECPROPMain}.

\begin{proposition}\label{PROPMain}
There are constants $\epsilon_0>0$, $c>0$, and $C>0$ such that for every $\epsilon \in (0,\epsilon_0)$, the following properties hold.
\begin{itemize}
\item[(i)] For every $t\in (0,L(\omega)]$, $\omega_1(t)>0$ and $|\omega_2(t)|\leq 2\epsilon$.
\item[(ii)] $\beta := \max_{t\in [0,L(\omega)]} | P(\omega(t))|   \leq C \, \epsilon^{3\bar{m}-1}$.
\item[(iii)] $\lambda <0$.
\item[(iv)] $P(\omega(t))>0$ for all $t\in (0,L(\omega))$.
\item[(v)] $\omega$ has a unique loop $\ell$, it satisfies
\[
\omega_2(s_{\ell}^-)\geq c \epsilon, \quad c \beta \epsilon^{-\bar{m}} \leq L(\ell)\leq C \beta^{1-1/m}, \quad \hbox{and} \quad \max_{t\in J_{\ell}} | P(\omega(t))|=\beta. \vspace{-4mm}
\]
\item[(vi)] $|\lambda| \beta^2 \geq c$.
\item[(vii)] If $t_* \in [0,L(\omega)]\setminus J_{\ell}$ is a local maximum of $t\mapsto P(\omega(t))$, then $|\lambda| P(\omega(t_*))^{1+1/\bar{m}}\leq C$.
\item[(viii)] $\int_{[0,L(\omega)]} |\dot{\theta}(t)| \, dt \leq 6\pi.$
\end{itemize}
\end{proposition}

\begin{figure}[H]
\begin{center}
\begin{tikzpicture}
\node[anchor=south west, inner sep=0] (image) at (0,0) {\includegraphics[width=6.8cm]{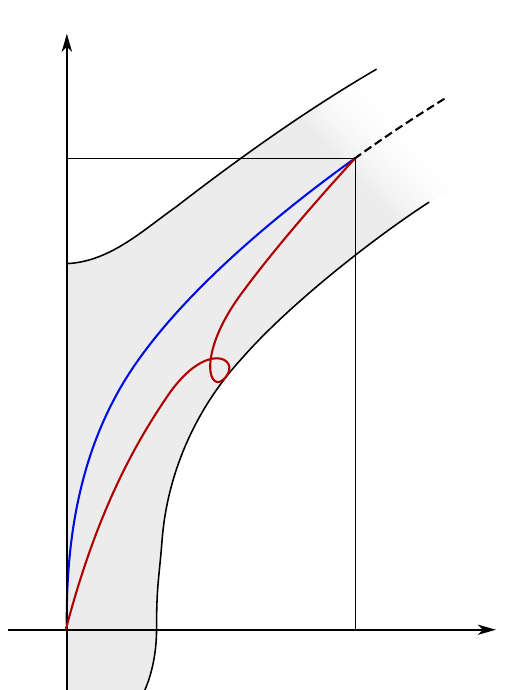}};

\begin{scope}[x={(image.south east)}, y={(image.north west)}]

    \node[black] at (0.3, 0.55) {$\bar\omega$};

    \node[black] at (0.48, 0.55) {$\omega$};

    \node at (0.7, 0.06) {$\epsilon^{m/2}$};

    \node at (0.66, 0.79) {$A_\epsilon$};

    \node at (0.09, 0.06) {$A_0$};

    \node at (0.35, 0.06) {$\beta^{1/2}$};

    \node at (0.1, 0.775) {$\epsilon$};

    \node at (0.12, 0.97) {$x_2$};

    \node at (0.97, 0.11) {$x_1$};

    \node at (0.91, 0.73) {$\{P=\beta\}$};

\end{scope}
\end{tikzpicture}
\caption{A drawing of $\bar{\omega}$ and $\omega$ \label{fig2}}
\end{center}
\end{figure}
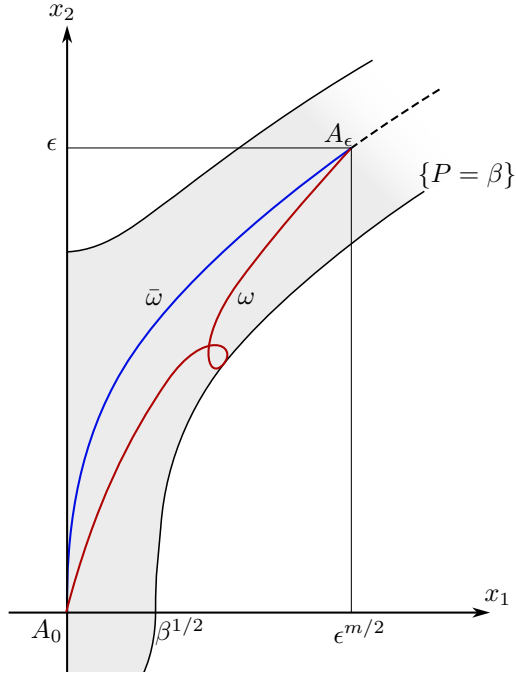

Assertion (iii) together with (iv) and (\ref{SYSomega}), allow one to show that the signed curvature of $\omega$ is negative, helping to visualize the curve $\omega$. The proof of the above proposition is challenging and
requires ruling out the possibility that the curve has multiple loops and crosses from one side of $\{P=0, x_1\geq 0\}$ to the other. The proof follows from Stokes theorem, the isoperimetric inequalities and geometric considerations.\\

We are now ready to demonstrate how the assertions of Proposition \ref{PROPMain} can be combined with Proposition \ref{PROPsublevelsets} to reach a contradiction. \\

\noindent {\bf Step 4:} The Contradiction.\\
We consider the simple Lipschitz curve $\nu : [0,s_{\ell}^-+L(\omega)-s_{\ell}^+] \rightarrow  [0,+\infty)\times \R$, which connects $A_0$ to $A_{\epsilon}$, defined as the concatenation
\[
\nu := \omega|_{ [0,s_{\ell}^-]} * \omega|_{ [s_{\ell}^+,L(\omega)]}.
\]
We observe that
\begin{eqnarray}\label{5janv1}
L\left(\bar{\omega}\right) \geq L(\omega) = L (\nu) + L(\ell).
\end{eqnarray}
Next, we set
\[
\rho := \left( \beta \epsilon^{-1}\right)^{\frac{m}{m-1}}
\]
and we note that, by Proposition \ref{PROPMain} (ii), $\rho \leq \epsilon^{3\bar{m}-1}$ for sufficiently small $\epsilon \in (0,\epsilon_0)$. We now distinguish two cases: either $P\circ \nu \leq \rho$, or this condition does not hold. \\

\noindent Case 1: $P\circ \nu \leq \rho$\\
As $\rho \leq \epsilon^{3\bar{m}-1}$, formula~\eqref{19dec1} in Proposition \ref{PROPsublevelsets} gives $L(\nu) \geq L(\bar{\omega}) - C(1) \rho^{1-1/m}$. Combining this with (\ref{5janv1}) yields
\[
 L(\ell) \leq L \left(\bar{\omega}\right) - L(\nu) \leq C(1) \rho^{1 - \frac{1}{m}} = C(1) \beta \epsilon^{-1}.
\]
Using the lower bound for $L(\ell)$ from Proposition \ref{PROPMain} (v), we then have $c\epsilon^{-\bar{m}} \leq C(1)\epsilon^{-1}$, which leads to a contradiction for sufficiently small $\epsilon \in (0,\epsilon_0)$, as $m\geq 5$.\\

\noindent Case 2: The condition $P\circ \nu \leq \rho$ is not satisfied\\
First, we note that if $P$ reaches a local maximum at $t_* \in [0,L(\omega)] \setminus J_\ell$ with $P(\omega(t_*))\geq \rho$, then assertion (vii) of Proposition \ref{PROPMain} implies
\[
|\lambda| \beta^2 \beta^{\frac{4-m}{m-1}} \epsilon^{-\frac{m+2}{m-1}}= |\lambda| (\beta \epsilon^{-1})^{\frac{m+2}{m-1}}  = |\lambda| \rho^{1+\frac{1}{\bar{m}}} \leq |\lambda| P\left(\omega\left(t_*\right)\right)^{1+\frac{1}{\bar{m}}}\leq C,
\]
which, by Proposition \ref{PROPMain} (vi), is impossible for sufficiently small $\epsilon \in (0,\epsilon_0)$, since $m\geq 5$. Therefore, as the condition $P\circ \nu \leq \rho$  is not satisfied, we have $P(\omega(s_{\ell}^-))=P(\omega(s_{\ell}^+)) > \rho$. This allows us to  define $t^-, t^+\in [0,L(\omega]$ by
\[
t^-:= \max \Bigl\{t\in [0,s^-] \, \vert \, P(\omega(t)) =\rho \Bigr\} \quad \mbox{and} \quad  t^+:= \min \Bigl\{t\in [s^+,L(\omega)] \, \vert \, P(\omega(t)) =\rho \Bigr\},
\]
and, additionally, for sufficiently small $\epsilon\in (0,\epsilon_0)$, we have
\begin{eqnarray}\label{33janv}
P(\omega(t)) > \rho \quad \forall t \in (t^-,s_{\ell}^-]\cup [s_{\ell}^+,t^+) \quad \mbox{and} \quad P(\omega(t)) < \rho \quad \forall t \in [0,t^-) \cup (t^+,L(\omega)].
\end{eqnarray}
Then, we claim that
\begin{eqnarray}\label{6janv1}
0 < \omega_2(t^+) - \omega_2(t^-)  \leq t^+-t^- \leq 2 \beta^{\frac{1}{2}} \epsilon^{1-\bar{m}} \quad \mbox{and} \quad 0 < \omega_2(t^-) < \omega_2(t^+) \leq 2 \epsilon,
\end{eqnarray}
provided that $\epsilon \in (0,\epsilon_0)$ is sufficiently small. To prove the left inequalities, we consider the set $\mathcal{S}:= \{0\leq P\leq \rho, x_1\geq 0\}$ whose boundary is the union  of: the vertical segment $S:=[A_0,(0,-\rho^{1/m})]$, the curve $\mathcal{C}_1:=\{P=0,\, x_1\geq 0\}$, and the curve $\mathcal{C}_2:=\{P=\rho,\, x_1\geq 0\}$. The smooth curve $\omega^-:=\omega|_{[0,t^-]}$ connects $\mathcal{C}_1$ to $\mathcal{C}_2$ and does not intersect $S$ for $t \in (0,t^-)$ (by Proposition \ref{PROPMain} (i)). As a result, the support $\mbox{spt}(\omega^-)$ of $\omega^-$ divides the set $\mathcal{S}$ into two connected components: $\mathcal{S}_1$, which is bounded by $S$, the segment of $\mathcal{C}_2$ from $(0,-\rho^{1/m})$ to $\omega(t^-)$, and $\mbox{spt}(\omega^-)$; and $\mathcal{S}_2$, which is the complement of $\mathcal{S}_1$ within $\mathcal{S}\setminus \mbox{spt}(\omega^-)$. Since $\omega(L(\omega))=A_{\epsilon} \notin \mathcal{S}_1$, $P(\omega(t^+))=\rho$, and the curves $\omega^-$ and $\omega|_{ [t^+,L(\omega)]}$ do not intersect (because $\ell$ is the unique loop of $\omega$, as stated in Proposition \ref{PROPMain} (v)), it follows that $\omega(t^+) \in \mathcal{S}_2 \cap \mathcal{C}_2$, which gives $ \omega_2(t^+) > \omega_2(t^-)$. The inequality $\omega_2(t^+) - \omega_2(t^-)  \leq t^+-t^-$ holds because $|\dot{\omega}(t)|=1$ for all $t\in [0,L(\omega)]$, as  $\omega$ is parametrized by arc length. To prove the next inequality, we suppose for contradiction that $t^+-t^-> 2\alpha$ with $\alpha := \beta^{1/2}\epsilon^{1-\bar{m}}$. Hence, we have either $s_{\ell}^--t^- > \alpha$ or  $t^+ - s_{\ell}^+ > \alpha$. Assume $s_{\ell}^--t^- > \alpha$; the other case follows similarly. Set $\bar{s}:= s_{\ell}^--\alpha> t^-$. Since $\omega_2(s_{\ell}^-)\geq c \epsilon$ by Proposition \ref{PROPMain} (v), and $\omega_2$ is $1$-Lipschitz, we have $\omega_2(t) \geq c \epsilon - \alpha$ for all $t\in [\bar{s},s_{\ell}^-]$.
Thus, since $\alpha=o(\epsilon)$ by Proposition \ref{PROPMain} (ii), for sufficiently small $\epsilon\in (0,\epsilon_0)$, we have $\omega_2(t) \geq c \epsilon/2$ for all $t\in [\bar{s},s_{\ell}^-]$. As a consequence, since $P(\omega(t)) >\rho>0$ on $(\bar{s},s_{\ell}^-]$ (by (\ref{33janv})), we have $\omega_{1}(t)\geq c' \epsilon^{\bar{m}}$ for all $t\in (\bar{s},s_{\ell}^-]$, where $c':=(c/2)^{1/2}$. By applying (\ref{SYSomega}), (\ref{33janv}), and assertion (viii) of Proposition \ref{PROPMain}, we obtain
\[
6\pi\geq \int_{\bar{s}}^{s_{\ell}^-} | \dot{\theta}(t)| \, dt =  \int_{\bar{s}}^{s_{\ell}^-}4  |\lambda| \omega_1(t)P(\omega(t))\, dt
\geq 4  \epsilon^{\frac{3}{2}}  |\lambda | c' \epsilon^{\bar{m}} \rho= 4 c' |\lambda | \beta^{\frac{m}{m-1}} \epsilon^{\frac{m+3}{2}-\frac{m}{m-1}}.
\]
By assertions (ii) and (vi) of Proposition \ref{PROPMain}, we know that $\beta\leq C \epsilon^{3\bar{m}-1}$ and $|\lambda| \beta^2 \geq c$. Hence, as $m/(m-1)<2$, the above inequality shows that the quantity
\[
\beta^{\frac{m}{m-1}-2} \epsilon^{\frac{m+3}{2}-\frac{m}{m-1}} \geq \epsilon^{(3\bar{m}-1)(\frac{m}{m-1}-2)} \epsilon^{\frac{m+3}{2}-\frac{m}{m-1}} = \epsilon^{\frac{-2m^2+8m-7}{2(m-1)}}
\]
is bounded from above for all $\epsilon>0$ sufficiently small. Since the exponent of $\epsilon$ is negative, this leads to a contradiction. The left inequalities in (\ref{6janv1}) are thus established. The right inequalities follow for sufficiently small $\epsilon \in (0,\epsilon_0)$, from the fact that $\omega_2(s_{\ell}^-)\geq c \epsilon$ (Proposition \ref{PROPMain} (v)), the $1$-Lipschitz continuity of $\omega_2$,  the upper bound on $t^+-t^-$ just derived, and the upper bound for $\omega_2$ from Proposition \ref{PROPMain} (i).

We now set $s^-:=\omega_2(t^-)>0,$ and $s^+:=\omega_2(t^+)\leq 2\epsilon$, and consider the concatenated curve
\[
\bar{\nu} := \omega|_{[0,t^-]} * \Gamma_\rho([s^-,s^+]) * \omega|_{[t^+,L(\omega)]}.
\]
From the right property in (\ref{33janv}) and inequality (\ref{19dec1}) in Proposition \ref{PROPsublevelsets} (since  $\rho \leq \epsilon^{3\bar{m}-1}$), it follows that $L(\bar{\nu}) \geq L(\bar{\omega}) - C(1) \rho^{1-1/m}$. Combining this with (\ref{5janv1}), we deduce
\begin{eqnarray}
\label{eq:startAPP}
 L(\ell) \leq L(\bar{\omega}) - L(\nu) \leq L(\bar{\nu}) - L(\nu)  + C(1)\rho^{1 - \frac{1}{m}}.
\end{eqnarray}
Using (\ref{PROPsublevelsetslast}) from Proposition \ref{PROPsublevelsets} and the left inequalities in (\ref{6janv1}), we have
\begin{align*}
L(\bar{\nu}) -L(\nu) & \leq  L(\Gamma_\rho ([s^-,s^+])) - L (\nu|_{ [t^-,s_{\ell}^-]}* \nu|_{ [s_{\ell}^+,t^+]}) \\
& \leq  L(\Gamma_\rho ([s^-,s^+])) - L([\Gamma_\rho(s^-),\Gamma_\rho (s^+)]) \\
& \leq \frac{ \bar{m}^2}{2} \left(\bar{m} -1\right) (s^+)^{m-3} (s^+-s^-)^2  \leq  2^{m-5}m^2 (m-2) \beta \epsilon^{-1}.
\end{align*}
Using the lower bound for $L(\ell)$ from Proposition \ref{PROPMain} (v) and (\ref{eq:startAPP}), we infer that $\epsilon^{-\bar{m}} \leq C'\epsilon^{-1}$ for some constant $C'>0$. Since $m\geq 5$, this leads to a contradiction for sufficiently small $\epsilon \in (0,\epsilon_0)$. This ends the proof of Theorem~\ref{MainTHM}.

\section{Proof of Proposition \ref{PROPMain}}\label{SECPROPMain}

Recall that $\omega : I_{\omega} \rightarrow \R^2$, where $I_{\omega}:=[0,L(\omega)]$, is a fixed curve satisfying (\ref{OptimPlane}) and (\ref{SYSomega}), and that $\bar{\omega}:=\bar{\omega}_{\epsilon}$ represents the projection of $\bar{\gamma}_{\epsilon}:=\bar{\gamma}|_{ [0,\epsilon]}$ onto the $(x_1,x_2)$-plane. By assumption, we have $L(\omega)\leq L(\bar{\omega})$. Recall that $m\geq 5$ is an odd integer, $\bar{m}:=m/2$, and that a loop of $\omega$ is a curve $\ell$ defined as the restriction of $\omega$ to some interval $J_{\ell}=[s_{\ell}^-,s_{\ell}^+] \subset I_{\omega}$, with $s_{\ell}^{+}\neq s_{\ell}^{-}$, such that $\omega(s_{\ell}^-)= \omega(s_{\ell}^+)$. We always denote by $J_{\ell}=[s_{\ell}^-,s_{\ell}^+] \subset I_{\omega}$ the interval associated with a loop $\ell$ of $\omega$, and we define $\mbox{int}(J_{\ell})=(s_{\ell}^-,s_{\ell}^+)$. Additionally, a loop $\ell$ is said to be simple if $\omega$ is injective on $[s_{\ell}^-,s_{\ell}^+)$.

Before beginning the proof of Proposition \ref{PROPMain}, we provide, in the next two sections, preliminary reminders about Stokes' theorem, the isoperimetric inequality, and Gauss--Bonnet formula for curves in the plane.

\subsection{Reminders on Stokes' theorem and the isoperimetric inequality}\label{SECReminders}

A Lipschitz curve is a curve $\eta:[0,\tau] \rightarrow \R^2$ that is Lipschitz continuous. The support of $\eta$, denoted by $\mbox{spt}(\eta)$, is defined as $\mbox{spt}(\eta):=\eta([0,\tau])$. We say that $\eta$ is closed if it satisfies $\eta(0)=\eta(\tau)$. For a Lipschitz closed curve $\eta$, the winding number (or index) of a point $x \in \R^2 \setminus \mbox{spt}(\eta)$ with respect to $\eta$ is denoted by $\mbox{ind}(x,\eta) \in \Z$. By convention, $\mbox{ind}(0,\eta)=+1$ when $\eta(t)=(\cos t, \sin t)$ for $t\in [0,2\pi]$. For each $k\in \Z$, we set
\[
\mathcal{E}_k(\eta) := \Bigl\{ x \in \R^2 \setminus \mbox{spt}(\eta) \, \vert \, \mbox{ind}(x,\eta) = k \Bigr\},
\]
and we define $\mathcal{E}\subset \R^2 \setminus \mbox{spt}(\eta)$ as the bounded open set
\[
\mathcal{E}(\eta) := \bigcup_{k\in \Z, k\neq 0} \mathcal{E}_k(\eta).
\]
Recalling that $Q$ is defined by $Q(x)=4x_1P(x)=4x_1(x_1^2-x_2^m)$, we define the weighted area of $\eta$ as
\[
A(\eta) := \sum_{k\in \Z} k \, \mathcal{L}_Q\left(\mathcal{E}_k(\eta)\right),
\]
where
\[
\mathcal{L}_Q\left(\mathcal{E}_k(\eta)\right) := \int_{\mathcal{E}_k(\eta)} Q(x) \, dx \qquad \forall k \in \Z.
\]
In the following result, assertion (i) follows from Stokes' theorem and (ii) is a consequence of Rad\'o's inequality in the plane. For two Lipschitz curves $\eta:[0,\tau] \rightarrow \R^2$ and $\eta':[0,\tau'] \rightarrow \R^2$ such that $\eta(\tau)=\eta'(0)$, we use the notation $\eta *\eta'$ to denote their concatenation on $[0,\tau+\tau']$. The reversed curve $\check{\omega}:[0,\epsilon] \rightarrow \R^2$ is defined by $\check{\omega}(t):= \bar{\omega}(\epsilon-t)$ for $t\in [0,\epsilon]$.

\begin{lemma}\label{LEMStokesIso}
Let $\eta:[0,\tau] \rightarrow \R^2$ be a Lipschitz curve. 
\begin{itemize}
\item[(i)] Assume that $\eta(0), \eta(\tau) \in \mbox{\rm spt}(\bar{\omega})$. Let $\check{\omega}^\eta$ denote the segment of $\bar{\omega}$ connecting $\eta(\tau)$ to $\eta(0)$. Then, the concatenation $\eta * \check{\omega}^\eta$ forms a closed curve and we have:
\[
A\left(\eta * \check{\omega}^\eta \right)  = \int_{\eta} P(x)^2 \, dx_2.
\]
In particular, for $\eta=\omega$ we have
\[
A(\omega * \check{\omega})  = \int_{\omega} P(x)^2 \, dx_2 = 0.
\]
\item[(ii)] If $\eta$ is closed, the weighted area $A(\eta)$ satisfies the inequality
\[
4\pi |A(\eta)| \leq \sup_{x\in \mathcal{E}(\eta)} |Q(x)| \, L(\eta)^2.
\]
\end{itemize}
\end{lemma}

\begin{proof}
Let $\eta:[0,\tau] \rightarrow \R^2$ be a Lipschitz curve. If $\eta(0), \eta(\tau) \in \mbox{spt}(\bar{\omega})$, then the concatenation $\eta * \check{\omega}^\eta$ forms a closed curve. By Stokes' theorem 
and noting that $P$ vanishes along $\check{\omega}^\eta$, we have
\[
A\left(\eta * \check{\omega}^\eta \right) = \sum_{k\in \Z, k\neq 0} k \int_{\mathcal{E}_k\left(\eta * \check{\omega}^\eta \right)} \frac{\partial P^2}{\partial x_1} (x) \, dx = \int_{\eta * \check{\omega}^\eta} P^2 \, dx_2 = \int_{\eta} P^2 \, dx_2,
\]
which proves the first part of (i). For $\eta$ closed, denoting by $\mathcal{L}$ the Lebesgue measure in the plane and using the definition of $A(\eta)$, we have
\[
|A(\eta)| \leq \sum_{k\in\Z} |k| \, |\mathcal{L}_Q(\mathcal{E}_k(\eta))| \leq \sup_{x\in \mathcal{E}(\eta)}|Q(x)| \, \sum_{k\in\Z} |k| \, |\mathcal{L}(\mathcal{E}_k(\eta))| \leq \sup_{x\in \mathcal{E}(\eta)}|Q(x)| \, \frac{L(\eta)^2}{4\pi},
\]
where the last inequality follows from Rad\'o's isoperimetric inequality in the plane (see \cite{rado47} and formula (1.9) in \cite{osserman78}), which proves (ii).
\end{proof}

\subsection{Curves with curvature of constant sign in the plane}\label{SECGaussBonnet}

We begin by recalling the definition of the signed curvature of a smooth curve. Let $\eta:[0,\tau]\rightarrow \R^2$ be a smooth curve parametrized by arc length. For every $t\in [0,\tau]$, the signed curvature of $\eta$ at $t$ is the unique value $\kappa(t) \in \R$ such that
\[
\ddot{\eta}(t) = \kappa(t) n(t),
\]
where $n(t)$ is obtained by rotating the nonzero vector $\dot{\eta}(t)$ counterclockwise by an angle $\pi/2$. If $\alpha : [0,\tau] \rightarrow \R$ is a smooth function satisfying $\dot{\eta}(t)=(\cos\alpha(t),\sin \alpha(t))$ for all $t\in [0,\tau]$ (such a function always exists), then the signed curvature is given by
\begin{eqnarray}\label{GaussBonnet1}
\kappa(t) = \dot{\alpha}(t) \qquad \forall t \in [0,\tau].
\end{eqnarray}
Next, we define a piecewise smooth continuous curve $\eta:[0,\tau] \rightarrow \R^2$ as a continuous curve for which there exist times $0=\tau_0 < \tau_1 < \cdots <\tau_N = \tau$ such that each restriction $\eta|_{ [\tau_i,\tau_{i+1}]}$, for $i=0,\ldots, N-1$, is smooth. For each $i=0, \ldots, N-1$, we denote by $\dot{\eta}(\tau_i^-)$ (resp. $\dot{\eta}(\tau_i^+)$) the left (resp. right) derivative of $\eta$ at $\tau_i$, with the conventions $\dot{\eta}(\tau_0^-):=\dot{\eta}(\tau_N)$ and $\dot{\eta}(\tau_N^+):=\dot{\eta}(\tau_0)$. We then consider such a curve $\eta$, assuming that it is closed, parametrized by arc length, and simple, meaning that the restriction $\eta|_{ [0,\tau)}$ is injective. By the Jordan curve theorem, $\eta$ divides the plane into two open sets: a bounded domain $\mathcal{D}(\eta)$ and its complement. In this setting, we have
\begin{eqnarray}\label{EQoriented}
\mathcal{D}(\eta) = \mathcal{E}_{k}(\eta) = \mathcal{E}(\eta),
\end{eqnarray}
where $k=1$ if $\eta$ is positively oriented and $k=-1$ otherwise. Recall that if $v,w$ are two nonzero vectors in $\R^2$, the oriented angle from $v$ to $w$, denoted by $\mbox{ang}(v,w)$, is defined to lie in the interval $(-\pi,\pi)$, with a positive sign if $(v,w)$ forms an oriented basis of $\R^2$, and a negative sign otherwise. For any nonzero vectors $v,w \in \R^2$, let $\mbox{Tan}^1(v,w)$ (resp. $\mbox{Tan}^{-1}(v,w)$) denote the open set of nonzero vectors $u \in \R^2$ such that $\mbox{ang}(v,u) \in (\mbox{ang}(v,w),\pi)$ (resp. $\mbox{ang}(v,u) \in (-\pi,\mbox{ang}(v,w))$). The relation (\ref{EQoriented}) implies the following result.

\begin{lemma}\label{LEM2025}
Let $\eta:[0,\tau] \rightarrow \R^2$ be a piecewise smooth continuous curve which  is closed, simple, parametrized by arc length, and satisfies (\ref{EQoriented}) with $k=\pm 1$. Then, for any $t\in [0,\tau]$ and every $u\in \mbox{\rm Tan}^k(\dot{\eta}(t^-),\dot{\eta}(t^+))$, we have
\[
\eta(t) + s k \, u \in  \mathcal{E}(\eta) \qquad \forall s >0 \mbox{ small}.
\]
\end{lemma}

Let us now assume that $\eta$ is positively oriented. For each  $i=0, \ldots, N-1$, if $\dot{\eta}(\tau_i^-) \neq \dot{\eta}(\tau_i^+)$, we define the discontinuity of the curvature at $\tau_i$ as the oriented angle $\delta_i = \mbox{ang}(\dot{\eta}(\tau_i^-),\dot{\eta}(\tau_i^+))$.  However, if  $\dot{\eta}(\tau_i^-) = -\dot{\eta}(\tau_i^+)$, meaning that $\eta$ has a cusp at $\eta(\tau_i)$, we follow the convention in \cite[Chapter 6]{spivak99}: if the cusp points toward $\mathcal{D}(\eta)$, we set $\delta_i:=-\pi$, otherwise, we set $\delta_i=\pi$. In any case, the discontinuity of the curvature $\delta_i$ belongs to $[-\pi,\pi]$. The Gauss--Bonnet formula then states
\begin{eqnarray}\label{GaussBonnet}
2 \pi = \sum_{i=0}^{N-1} \int_{\tau_i}^{\tau_{i+1}} \kappa(t) \, dt +  \sum_{i=0}^{N-1} \delta_i.
\end{eqnarray}
From this, we can easily deduce the following result.

\begin{lemma}\label{LEMGaussBonnet1}
Let $\eta:[0,\tau] \rightarrow \R^2$ be a piecewise smooth continuous curve associated with times $0=\tau_0 < \tau_1 < \cdots < \tau_N = \tau$, which  is closed, simple, and parametrized by arc length. Let $\sigma=\pm 1$ be fixed.
\begin{itemize}
\item[(i)] If for every $i=0, \ldots, N-1$, the smooth signed curvature $\kappa$ of the segment $\eta|_{ [\tau_i,\tau_{i+1}]}$ satisfies $\sigma \kappa\geq 0$, and the discontinuity of the curvature $\delta_i$ at $\tau_i$ satisfies $\sigma \delta_i\in  [0,\pi]$, then the set $\mathcal{D}(\eta)$ is convex, $\mathcal{D}(\eta) = \mathcal{E}_{\sigma}(\eta)= \mathcal{E}(\eta)$, and for every $i=0,\ldots, N-1$, $\sigma \delta_i = \sigma \mbox{\rm ang}(\dot{\eta}(\tau_i^-),\dot{\eta}(\tau_i^+)) \in [0,\pi)$.
\item[(ii)] If there exist indices $i_1 \neq i_2$ in $\{0,\ldots, N-1\}$ such that
\[
\sum_{i=0}^{N-1} \int_{\tau_i}^{\tau_{i+1}} \sigma\kappa(t)\, dt \geq 0 \quad \mbox{and} \quad \sigma \delta_i\in [0,\pi] \quad \forall i \in \{0,\ldots, N-1\} \setminus \{i_1,i_2\},
\]
then $\mathcal{D}(\eta) = \mathcal{E}_{\sigma}(\eta) = \mathcal{E}(\eta)$.
\end{itemize}
\end{lemma}
\begin{proof}
To prove (i), we observe that under the given assumption, the set $\mathcal{D}(\eta)$ admits a local supporting line at each point of its boundary $\partial \mathcal{D}(\eta)$. Specifically, for every $x\in \partial \mathcal{D}(\eta)$, there exist a neighborhood $U$ of $x$ and a closed half-plane $P$, bounded by a line $L$, such that $x\in L$ and $U \cap \mathcal{D}(\eta) \subset P$. By Tietze's theorem~\cite{valentine64}, this implies that $\mathcal{D}(\eta)$ is convex and that the oriented angles $\sigma \delta_0, \ldots, \sigma \delta_{N-1}$ lie within $[0,\pi)$. Moreover, if $\sigma=1$ then $\eta$ is positively oriented, which, by (\ref{EQoriented}), implies $\mathcal{D}(\eta) = \mathcal{E}_{1}(\eta)= \mathcal{E}(\eta)$. Conversely, if $\sigma=-1$, then $\eta$ is negatively oriented, leading to $\mathcal{D}(\eta) = \mathcal{E}_{-1}(\eta)= \mathcal{E}(\eta)$ (by (\ref{EQoriented})).

To prove (ii), using (\ref{EQoriented}), it suffices to show that $\eta$ is positively oriented when $\sigma=1$. Assume $\sigma=1$ and suppose, for the sake of contradiction, that $\eta$ is negatively oriented. Define the curve $\tilde{\eta}:[0,\tau] \rightarrow \R^2$ by $\tilde{\eta}(t):=\eta(\tau -t)$ for all $t\in [0,\tau]$. Let $\tilde{\delta}_0, \ldots, \tilde{\delta}_{N-1}\in [-\pi,\pi]$ denote the discontinuities of the signed curvature at $t=0$, $t=\tau-\tau_{N-1}, \ldots, t=\tau-\tau_1$. The curve $\tilde{\eta}$ is positively oriented by construction. Furthermore, by assumption,  the integral of the signed curvature $\tilde{\kappa}$ of $\tilde{\eta}$ over the set $[0,\tau-\tau_{N-1}] \cup \cdots \cup [\tau-\tau_1,\tau]$ is nonpositive, and the discontinuities satisfy $\tilde{\delta}_i \in [-\pi,0]$ for all $i\in \{0,\ldots, N-1\}\setminus \{i_1,i_2\}$. By the Gauss--Bonnet formula, we have
\[
2\pi = \sum_{i=0}^{N-1} \int_{\tau_i}^{\tau_{i+1}} \tilde{\kappa}(t)\, dt + \sum_{i=0, i\neq i_1,i_2}^{N-1}  \tilde{\delta}_i \, + \tilde{\delta}_{i_1} + \tilde{\delta}_{i_2} \leq  \tilde{\delta}_{i_1} + \tilde{\delta}_{i_2} \leq 2\pi.
\]
Hence, equality holds, which forces $\tilde{\kappa}\equiv 0$ and $\tilde{\delta}_i=0$ for all $i\neq i_1, i_2$. Since $\eta$ is simple, such a configuration cannot arise. We obtain a contradiction.
\end{proof}

The following result, derived from the Gauss--Bonnet formula, will also be instrumental in the proof of Proposition \ref{PROPMain}.

\begin{lemma}\label{LEMGaussBonnet}
Let $\eta : [0,\tau] \rightarrow \R^2$ be a smooth curve parametrized by arc length, and let $\sigma=\pm 1$ such that the signed curvature $\kappa$ of $\eta$ satisfies $\sigma \kappa \geq 0$.  If $\eta|_{ (0,\tau)}$ is injective and $\eta|_{ (0,\tau)}$ does not cross the line $\eta(0)+\R\,\dot{\eta}(0)$, then the point $\eta(\tau)$ does not belong to the open ray $\eta(0)+(0,+\infty)\,\dot{\eta}(0)$.
\end{lemma}

\begin{proof}
We may assume without loss of generality that $\eta(0)=(0,0)$, $\dot{\eta}(0)=(1,0)$, and $\eta_2 \geq 0$ on $[0,\tau]$. Note that under these assumptions, we necessarily have $\kappa\geq 0$. Moreover, since $\eta|_{ (0,0)}$ does not cross the line $\eta(0)+\R\,\dot{\eta}(0)$, we have $\int_{0}^{\tau} \kappa(t)\, dt >0$. Suppose, for the sake of contradiction that $\eta(\tau) \in (0,+\infty) (1,0)$. Consider the curve $\tilde{\eta}$, defined as the concatenation of the line segment $[\eta(0),\eta(\tau)]$ with the curve $\eta|_{[0,\tau]}$ traversed backward. This curves is positively oriented. Applying  the Gauss--Bonnet formula, we obtain
\[
2\pi = - \int_{0}^{\tau} \kappa(t)\, dt + \delta_0 + \pi,
\]
where $\delta_0\in [-\pi,\pi]$ represents the discontinuity of curvature of $\tilde{\eta}$ at the point $\eta(\tau)$. Since $\int_{0}^{\tau}\kappa (t) \, dt>0$, this leads to a contradiction.
\end{proof}

\subsection{Preliminary observations on the curve $\omega$}\label{PROPPrel}

The curve $\omega : I_{\omega}=[0,L(\omega)] \rightarrow \R^2$ is an analytic curve parametrized by arc length, joining $A_0=\bar{\omega}(0)$ to $A_{\epsilon}=\bar{\omega}(\epsilon)$. Its length $L(\omega)$ satisfies
\begin{eqnarray}\label{11dec0}
L(\omega) \leq L(\bar{\omega}) = \int_0^{\epsilon} \left|\dot{\bar{\omega}}_{\epsilon}(t)\right| \, dt = \int_0^{\epsilon} \left( 1+\bar{m}^2 t^{m-2}\right)^{\frac{1}{2}} \, dt = \epsilon + \frac{\bar{m}^2 \epsilon^{m-1}}{2(m-1)} + o(\epsilon^{m-1}),
\end{eqnarray}
for small $\epsilon>0$. The following lemma provides several results that will be used repeatedly throughout the proof of Proposition \ref{PROPPrel}. Notably, the property stated in (v) is reminiscent of an estimates used by Liu and Sussmann in \cite[Lemma 1 p. 14]{ls95} to establish the minimality of the (smooth) singular horizontal path they considered.

\begin{lemma}\label{PROPPrelLEM1}
There are constants $\epsilon_0, C>0$ such that for every $\epsilon \in (0,\epsilon_0)$, the following hold.
\begin{itemize}
\item[(i)] For every $t\in (0,L(\omega)]$, $\omega_1(t)>0$.
\item[(ii)] $\theta(0)\in (-\pi/2,\pi/2)$.
\item[(iii)] $\max_{t\in [0,L(\omega)]} P(\omega(t)) >0$.
\item[(iv)] For every $t\in I_{\omega}$, $\omega_1(t) \leq 2 \epsilon^{\bar{m}}$ and $|\omega_2(t)|\leq 2\epsilon$.
\item[(v)] $\beta := \max_{t\in I_{\omega}} |P(\omega(t))| \leq C \, \epsilon^{3\bar{m}-1} = o(\epsilon^m)$ as $\epsilon \rightarrow 0$.
\item[(vi)] The function $t\in I_{\omega} \mapsto \omega(t)$ is not injective.
\item[(vii)] $\lambda \neq 0$.
\item[(viii)] Any loop $\ell$ of $\omega$ satisfies $L(\ell) \leq C\beta^{1-\frac{1}{m}} = o(\epsilon^{\bar{m}})$ as $\epsilon \rightarrow 0$.
\end{itemize}
\end{lemma}

\begin{proof}
To prove (i), assume that there is an interval $J=[t_0,t_1] \subset I_{\omega}$ such that $
\omega_1(t_0) = \omega_1(t_1) = 0$ and $\omega_1(t) \leq 0$ for all $t \in J$. Then the curve $\hat{\omega}: I_{\omega} \rightarrow \R^2$, defined as
\[
\hat{\omega}(t) := \left\{
\begin{array}{cl}
\omega(t) & \mbox{ if } t \notin J\\
\left( -\omega_1(t),\omega_2(t)\right) & \mbox{ if } t \in J
\end{array}
\right.
\]
has the same length as $\omega$. Furthermore, it satisfies
\[
\int_{\hat{\omega}} P^2\, dx_2 = \int_0^{L(\omega)} P( \hat{\omega}(t))^2 \dot{\hat{\omega}}_2(t) \, dt =  \int_0^{L(\omega)} \left(\hat{\omega}_1(t)^2 - \hat{\omega}_2(t)^m \right)^2 \dot{\hat{\omega}}_2(t) \, dt = \int_{\omega} P^2\, dx_2=0.
\]
As a consequence, if $\omega(t)<0$ for some $t\in I_{\omega}$ then we can define a nonanalytic curve that minimizes the length from $\bar{\omega}(0)$ to $\bar{\omega}(\epsilon)$, which is a contradiction. If $\omega_1(t)=0$ for some $t\in (0,L(\omega)$, then $\cos \theta(t)=0$. In that case, as well as when $t=0$ and $\cos \theta(0)=0$, the uniqueness of the solution to (\ref{SYSomega}) with fixed initial conditions implies that the curve $s \mapsto \omega (t+s)$ must coincide with the straight line $s \mapsto (0,\omega_2(t)+(s-t)L(\omega)\sin \theta(t))$. This leads to a contradiction, thereby completing the proofs of (i) and (ii). By (ii) and (\ref{SYSomega}), we have $P(\omega(t)) >0$ for $t>0$ small, which proves (iii).

To prove (iv), it is sufficient to show that there is no $t\in I_{\omega}$ such that $\omega_2(t)=2\epsilon$ or $\omega_1(t)=2\epsilon^{\bar{m}}$. Suppose there exists $t\in I_{\omega}$ such that $\omega_2(t)=2\epsilon$. Then we have
\[
L(\omega) \geq L([\omega(0),\omega(t)]) + L([\omega(t),\omega(\epsilon)]) = \sqrt{\omega_1(t)^2+4\epsilon^2} + \sqrt{(\epsilon^{\bar{m}}-\omega_1(t))^2+\epsilon^2} \geq 3 \epsilon,
\]
which contradicts (\ref{11dec0}) for sufficiently small $\epsilon>0$. Suppose now that there exists $t\in I_{\omega}$ such that $\omega_1(t)=2\epsilon^{\bar{m}}$, then as above we have
\[
L(\omega) \geq  \sqrt{4\epsilon^m+\omega_2(t)^2} + \sqrt{\epsilon^m+(\epsilon-\omega_2(t))^2} \geq \epsilon \sqrt{1+\epsilon^{m-2}},
\]
where the last inequality follows form the fact that the function $z\in \R\mapsto  \sqrt{4\epsilon^m+z^2} + \sqrt{\epsilon^m+(\epsilon-z)^2}$ has a minimum at $z=2\epsilon/3$. The inequality contradicts (\ref{11dec0}) for sufficiently small $\epsilon>0$.

To prove (v), we start by computing
\[
J:= \int_0^{L(\omega)} P(\omega(t))^2 \, dt.
\]
Since $\omega$ joins $\bar{\omega}(0)$ to $\bar{\omega}(\epsilon)$ and $\int_{\omega}P^2dx_2=0$, we have (noting that $1-\dot{\omega}_2\geq 1-|\dot{\omega}| \geq 0$)
 \begin{eqnarray}\label{11dec1}
 J = \int_0^{L(\omega)} \left(1-\dot{\omega}_2(t)\right)  P(\omega(t))^2 \, dt \leq \beta^2  \int_0^{L(\omega)} \left(1-\dot{\omega}_2(t)\right) \, dt = \beta^2 \left( L(\omega)- \epsilon\right).
 \end{eqnarray}
 Furthermore, the derivative of $t\in I_{\omega} \mapsto P(t) := P(\omega(t))$ satisfies (using (iv))
 \[
 | \dot{P}(t)| = | 2\omega_1(t)\dot{\omega}_1(t)- m\omega_2^{m-1}(t)\dot{\omega}_2(t)| \leq 2|\omega_1(t)| + m| \omega_2(t)|^{m-1} \leq 4\epsilon^{\bar{m}} + m (2\epsilon)^{m-1} \leq 5 \epsilon^{\bar{m}},
 \]
 provided $\epsilon>0$ is sufficiently small. We infer that (where $\mathcal{L}^1$ denotes the Lebesgue measure)
 \[
 J \geq \int_{\{|P|\geq \beta/2\}} P(t)^2 \, dt \geq \frac{\beta^2}{4}  \mathcal{L}^1(\{|P|\geq \beta/2\}) \geq \frac{\beta^2}{4} \frac{\beta}{5\epsilon^{\bar{m}}} = \frac{\beta^3}{20\epsilon^{\bar{m}}}.
 \]
The inequality in (v) follows directly from (\ref{11dec1}) and the inequality $L(\omega)-\epsilon \leq m^2\epsilon^{m-1}$, as established in (\ref{11dec0}) for sufficiently small $\epsilon>0$. We deduce that $\beta=o(\epsilon^m)$ since $m\geq 5$.

 To prove (vi), suppose, for contradiction, that $\omega$ is injective. Then, by Lemma \ref{LEMStokesIso} (i), we have $A(\omega*\check{\omega})=0$. Since both $\omega$ and $\bar{\omega}$ are injective, the winding number of any point with respect to the closed curve $\eta := \omega*\check{\omega}$ is $\pm 1$ or $0$. Consequently,  the bounded open set $\mathcal{E}(\omega*\check{\omega})$ is the union of the two disjoint sets $\mathcal{E}_{-1}(\eta)$ and $\mathcal{E}_{1}(\eta)$. From assertion (i), the function $Q(\omega(t))$ has the same sign as $P(\omega(t))$. Furthermore, since $P(\bar{\omega})\equiv 0$, the connected components of $\mathcal{E}_{-1}(\eta)$ are contained in $\{Q<0\}$, while the connected components of $\mathcal{E}_{1}(\eta)$ are contained in $\{Q>0\}$. As $\omega$ and $\bar{\omega}$ are not identical, one of the sets $\mathcal{E}_{-1}(\eta)$ or $\mathcal{E}_{1}(\eta)$ must be nonempty. This implies that $A(\eta)>0$, contradicting the earlier result that $A(\omega*\check{\omega})=0$. The contradiction proves that $\omega$ cannot be injective.

To prove (vii), we observe that if $\lambda =0$, then by (\ref{SYSomega}), $\omega$ is a straight line, and thus injective. however, this contradicts assertion (vi).

To prove (viii), we observe that the curve $\nu:= \omega|_{ [0,s_{\ell}^-]}*\omega|_{ [s_{\ell}^+,L(\omega)]}$ satisfies (\ref{CalcVar}). Applying \eqref{19dec1} in Proposition \ref{PROPsublevelsets} with $\rho=\beta$, we obtain
\[
L\left(\bar{\omega}\right) = L(\ell)+ L(\nu) \geq L(\ell) + L\left(\nu_{\epsilon}^{\rho}\right) \geq L(\ell) + L\left(\bar{\omega}\right) - C \rho^{1-\frac{1}{m}}.
\]
Rearranging terms gives the desired result. The relation $L(\ell) = o(\epsilon^{\bar{m}})$ follows from $m\geq 5$.
\end{proof}

The next result follows from the upper bound for $\beta$ given in Lemma \ref{PROPPrelLEM1}, combined with inequality (\ref{11dec0}).

\begin{lemma}\label{PROPPrelLEM2}
For every $K > 0$, there exist constants $\epsilon_0(K), C(K) >0$ such that for any $\epsilon \in (0,\epsilon_0(K))$ and $t\in I_{\omega}$, the following holds:
\[
\omega_2(t) \geq K\epsilon \quad \Longrightarrow \quad
\left\{ \begin{array}{rcl}
\omega_1(s) & \geq & C(K)\, \epsilon^{\bar{m}}\\
\omega_2(s) & \geq & C(K)\, \epsilon
\end{array}
\right.
\quad \forall s \in [t,L(\omega)].
\]
\end{lemma}

\begin{proof}
Let $K>0$ and $t\in I_{\omega}$ be such that $\omega_2(t)\geq K\epsilon$. Define $C(K):=K/2$, and assume  there exists $s\in [t,L(\omega)]$ such that $\omega_2(s)<C(K)\epsilon$. Since $\omega_2(0)=0$ and $\omega_2(L(\omega))=\epsilon$, we deduce that
\[
L(\omega) \geq L(\omega|_{[0,t]}) +  L(\omega|_{[s,L(\omega)]}) \geq \omega_2(t) + \epsilon - \omega_2(s) \geq \epsilon + C(K) \epsilon.
\]
This inequality contradicts (\ref{11dec0}) for sufficiently small $\epsilon>0$. Hence, for sufficiently small $\epsilon >0$, we must have $\omega_2(s)\geq C(K)\epsilon$ for all $s\in [t,L(\omega)]$. As a result, using the bound for $\beta$ from Lemma \ref{PROPPrelLEM1} (v), for all $s\in [t,L(\omega)]$, we obtain:
\[
\omega_1(s)^2 = \omega_2(s)^m + P(\omega(s)) \geq \omega_2(s)^m - |P(\omega(s))| \geq C(K)^m\epsilon^m- \beta \geq C(K)^m\epsilon^m-  C \, \epsilon^{3\bar{m}-1}.
\]
We conclude by noting that $3\bar{m}-1>m$.
\end{proof}

\subsection{Some obstructions to the minimality of $\omega$}\label{SECObs}

In the following result, we present a series of obstructions that arise from the minimality of $\omega$. Each of these scenarios is ruled out because, if any were to occur, we could use Stokes' theorem and/or the isoperimetric inequality, as stated in Lemma \ref{LEMStokesIso}, to construct a new curve satisfying (\ref{OptimPlane}) that is either shorter than $\omega$ or has the same length but fails to be analytic or identical to $\bar{\omega}$.

\begin{lemma}\label{LEMObs}
By taking $\epsilon_0>0$ from Lemma \ref{PROPPrelLEM1} smaller if necessary, none of the following situations can occur for any $\epsilon \in (0,\epsilon_0)$:
\begin{itemize}
\item[(i)] There exist a loop $\ell$ of $\omega$ and a Lipschitz closed curve $\eta:[0,\tau] \rightarrow \R^2$ that intersects the curve $\omega|_{ [0,s_{\ell}^-]} * \omega|_{ [s_{\ell}^{+},L(\omega)]}$ such that $L(\eta) \leq L(\ell)$ and $|A(\ell)| \leq |A(\eta)|$.
\item[(ii)] There exist $t_1<t_2 \in I_{\omega}$ and a Lipschitz closed curve $\eta:[0,\tau] \rightarrow \R^2$ that intersects the curve $\omega|_{ [0,t_1]} * [\omega(t_1),\omega(t_2)] * \omega|_{ [t_2,L(\omega)]})$ such that $L(\eta) \leq L(\omega|_{ [t_1,t_2]})- L([\omega(t_1),\omega(t_2)])$ and $|A(\omega|_{ [t_1,t_2]} * [\omega(t_2),\omega(t_1)])| < |A(\eta)|$.
\item[(iii)] There exist a simple loop $\ell$ of $\omega$, $t\in I_{\omega}\setminus \mbox{\rm int}(J_{\ell})$, $\sigma = \pm 1$, and $s>0$ such that $\sigma (P\circ \ell) \geq 0$ and $\omega(t)+\sigma (0,s) \in \mbox{\rm{spt}}(\ell)$.
\item[(iv)] There exist $t_1< t_2$ in $I_{\omega}$ such that $P(\omega(t_1))=P(\omega(t_2))=0$, $(P\circ \omega)|_{ (t_1,t_2)}<0$, and $\omega$ is injective on $[t_1,t_2)$.
\item[(v)] There exist a loop $\ell$ of $\omega$ and $t^*\in I_{\omega}\setminus \mbox{\rm int}(J_{\ell})$ such that $\max_{t\in J_{\ell}}|Q(\omega(t))| \leq Q(\omega(t^*))$.
\item[(vi)] There exist two loops $\ell_1$ and $\ell_2$ of $\omega$ such that $\mbox{\rm int}(J_{\ell_1}) \cap \mbox{\rm int}(J_{\ell_2})=\emptyset$ and $Q\circ \ell_1, Q\circ \ell_2 \geq 0$.
\end{itemize}
Moreover, for every $K>0$, there exists $\epsilon_0(K)>0$ such that  for every $\epsilon \in (0,\epsilon_0(K))$, the following situations cannot occur:
\begin{itemize}
\item[(vii)] There exist a loop $\ell$ of $\omega$ and $t^*\in I_{\omega}\setminus \mbox{\rm int}(J_{\ell})$ such that $\max_{t\in J_{\ell}} |Q(\omega(t))| \leq -Q(t^*)$ and $\omega_2(t^*)\geq K \epsilon$.
\item[(viii)] There exist two loops $\ell_1$ and $\ell_2$ of $\omega$ such that $\mbox{\rm int}(J_{\ell_1}) \cap \mbox{\rm int}(J_{\ell_2})=\emptyset$  and $\omega_2(s_{\ell_1}^-), \omega_2(s_{\ell_2}^-)  \geq K \epsilon$.
\end{itemize}
\end{lemma}

\begin{proof}
Assertion (i) follows from assertion (ii) with $t_1=s_{\ell}^{-}$ and $t_2=s_{\ell}^{+}$. To prove (ii), we consider $t_1<t_2 \in I_{\omega}$ and a Lipschitz closed curve $\eta:[0,\tau] \rightarrow \R^2$ satisfying the assumptions. By reversing the orientation of $\eta$ if necessary, we may assume that there exists $t^* \in [t_2,L(\omega)]$ such that $\eta(0)=\omega(t^*)$, and that $A(\eta)$ and $A(\alpha)$, with $\alpha:= \omega|_{ [t_1,t_2]} * [\omega(t_2),\omega(t_1)]$, have the same sign. Define, for any $r>0$, the curve $\eta^r:[0,\tau] \rightarrow \R^2$ by $\eta^r(t):=\eta(0)+r(\eta(t)-\eta(0))$. Next, choose $r\in (0,1]$ such that $A(\eta^r)=A(\alpha)$. By construction, $L(\eta^r)\leq L(\eta)$, and the concatenated path
\[
\zeta := \omega|_{ [0,t_1]} * [\omega(t_1),\omega(t_2)] * \omega|_{ [t_2,t^*]} * \eta^r *  \omega|_{ [t^*,L(\omega)]}
\]
connects $\omega(0)=\bar{\omega}(0)$ to $\omega(\epsilon)=\bar{\omega}(\epsilon)$. Furthermore, by Lemma \ref{LEMStokesIso} (i), it satisfies
\[
 \int_{\zeta} P^2 \, dx_2= \int_{\omega} P^2 \, dx_2 - \int_{\alpha} P^2 \, dx_2 + \int_{\eta^r} P^2 \, dx_2 = A(\eta^r)- A(\alpha)=0.
\]
Using the assumption $L(\eta) \leq L(\omega|_{ [t_1,t_2]})- L([\omega(t_1),\omega(t_2)])$, we compute
\[
L(\zeta)-L(\omega) = L([\omega(t_1),\omega(t_2)]) + L(\eta^r) -  L(\omega|_{ [t_1,t_2]}) \leq L(\eta^r) - L(\eta) \leq 0.
\]
Thus, $\zeta$ minimizes the length among all curves satisfying (\ref{OptimPlane}). Since $\zeta$ is not analytic and not identical to $\bar{\omega}$, this is a contradiction.

To prove (iii), let  $\ell$ be a simple loop of $\omega$, $t\in I_{\omega}\setminus J_{\ell}$, $\sigma =\pm 1$, and $s>0$ such that $\sigma P \geq 0$ over $\ell$ and $\omega(t)+ \sigma(0,s) \in \mbox{spt}(\ell)$. We treat only the case $\sigma=1$, the proof for $\sigma=-1$ follows in the same manner. By reversing the orientation of $\ell$ if necessary, we may assume that it is positively oriented. Since $\ell$ is simple, the weighted area of the loops $\ell$ and $\eta:=\ell - (0,s)$ are given by the integrals of $Q(x)=4x_1P(x)$ over $\mathcal{E}_1(\ell)$ and $\mathcal{E}_1(\eta)$, respectively. Observe that for every $x_1>0$, the function $x_2 \mapsto Q(x_1,x_2) = 4x_1 (x_1^2-x_2^m)$ is decreasing on $\R$. Consequently, since $Q\geq 0$ over $\ell$, we have
\[
|A(\ell)| = \int_{\mathcal{E}_1(\ell)} Q(x) \, dx < \int_{\mathcal{E}_1(\eta)} Q(x_1,x_2-(0,s)) \, dx =|A(\eta)|.
\]
By construction, $\omega(t) \in \mbox{spt}(\eta)$ and $L(\eta) = L(\ell)$. Therefore, the result follows from (i).

To prove (iv), we need to perform a reflection with respect to the set
\[
S:=(\{0\}\times (-\infty,0)) \cup \Bigl\{(x_1,x_2)\in \R^2 \, \vert \, P(x_1,x_2)= 0, x_1\geq 0\Bigr\}.
\]
Since $m\geq 5$, $S$ is a $1$-dimensional submanifold of $\R^2$ of class at least $C^2$. Consequently, the signed distance function $d_S$ to $S$, assumed to be positive on the open set $S^+=(0,+\infty)\times (0,-\infty)\cup \{P>0,x_2\geq 0\}$, is of class at least $C^2$ and solution to the eikonal equation $|\nabla d_S|=1$ in a neighborhood of $S$ (see \cite[Lemma 14.16]{gt83}). Thus, there is a ball $\mathcal{B}$ centered at $\bar{\omega}(0)=(0,0)$ such that for every $x\in S\cap \mathcal{B}$, $\nabla d_S(x)$ is the unit normal vector at $x$, pointing toward $S^+$. Moreover, for any small $s\in \R$, the point $x_s:=x+s\nabla d_S(x)$ belongs to $\mathcal{B}$ and satisfies $d_S(x_s)=s$ and $\nabla d_S(x_s)=\nabla d_S(x)$.  In particular, for distinct points $x, y$ in $S\cap \mathcal{B}$, the lines $\{x_s, s \in \R \}$ and $\{y_s, s \in \R \}$, which correspond to the orbits of the vector field $\nabla d_S$, do not intersect as long as both $x_s$ and $y_s$ remain within $\mathcal{B}$. For each $r\in (0,1]$, we define the reflection map $R_r:\mathcal{B} \rightarrow \R^2$ with respect to $S$ as
\[
R_r(x) := x- (1+r) \, d_S(x)\nabla d_S(x) \qquad \forall x \in \mathcal{B}.
\]
By shrinking $\mathcal{B}$ if necessary, $R_r$ is well-defined and constitutes a diffeomorphism onto its image that reverses orientation. Furthermore, since the closure $\bar{S}^+$ of $S^+$ is convex, the projection $\pi$ onto $\bar{S}^+$ is  well-defined, and we have $R_r(x):=x+(1+r)(\pi(x)-x)$ for any $x\in \R^2\setminus S^+$, which shows that $R_r$ is $1$-Lipschitz on $\R^2\setminus S^+$. Now, consider $t_1< t_2$ in $I_{\omega}$ satisfying the assumptions of (iv). If $\epsilon>0$ is sufficiently small, the Lipschitz curve $\omega_R:[t_1,t_2] \rightarrow \R^2$ defined by $\omega_R:=R_1\circ \omega|_{ [t_1,t_2]}$ is well-defined, and its support is contained in $\R^2\setminus S^+$. Moreover, by the above discussion, the mapping $\psi:(s,t) \mapsto \bar{\omega}(t)+s\bar{\nu}(t)$, where $\bar{\nu}(t):=|\dot{\omega}(t)|\nabla d_S(\bar{\omega}(t))=(1,-\bar{m}t^{\bar{m}-1})$, provides a diffeomorphism that enables the computation of the weighted area of $\omega *\hat{\omega}$. Here, $\hat{\omega}$ denotes the segment of $\omega$ joining $\omega(t_2)$ to $\omega(t_1)$. The Jacobian determinant of $\psi$ at $(s,t)$ is equal to $|\dot{\bar{\omega}}(t)|^2+s\langle \dot{\bar{\omega}}(t),\dot{\bar{\nu}}(t)\rangle$. Considering the winding number $k =\pm 1$ of a point in the open set $\mathcal{E}$ enclosed by the simple closed curve $\omega|_{ [t_1,t_2]} * \hat{\omega}$, we have
\[
A\left(\omega|_{ [t_1,t_2]} * \hat{\omega}\right) = k \int_{\mathcal{E}} Q(x)\, dx = k \int_{0}^{+\infty} \int_{N_t} Q(\psi(-s,t)) \left(|\dot{\bar{\omega}}(t)|^2- s\langle \dot{\bar{\omega}}(t), \dot{\bar{\nu}}(t)\rangle \right) \, ds \, dt,
\]
where, for each $t\in [0,+\infty)$, the interval $N_t\subset [0,+\infty)$ represents the set of $s\geq 0$ such that $\Psi(-s,t)\in \mathcal{E}$. The reflection map $R_1$ reverses the orientation and, by construction, the open set $\mathcal{E}_R$, enclosed by the simple closed curve $\omega_{R} * \hat{\omega}$, coincides with the set of $\psi(s,t)$ for $t\in [0,+\infty)$ and $s\in N_t$. Thus,
\[
A\left(\omega_{R} * \hat{\omega}\right) = - k \int_{\mathcal{E}_{R}} Q(x)\, dx = - k \int_{0}^{+\infty} \int_{N_t} Q(\psi(s,t)) \left(|\dot{\bar{\omega}}(t)|^2+ s\langle \dot{\bar{\omega}}(t), \dot{\bar{\nu}}(t)\rangle \right) \, ds \, dt.
\]
Since $Q$ is positive on $\mathcal{E}_R$ and negative on $\mathcal{E}$, we conclude that $A(\omega_R * \hat{\omega})$ and $A(\omega|_{ [t_1,t_2]} * \hat{\omega})$ have the same sign and satisfy
\begin{eqnarray}\label{22dec1}
\left|A(\omega_R * \hat{\omega})\right| - \left|A\left(\omega|_{ [t_1,t_2]} * \hat{\omega}\right)\right| = \int_0^{+\infty} \int_{N_t} \Delta (s,t) \left|\dot{\bar{\omega}}(t)\right|^2 \, ds \, dt,
\end{eqnarray}
where $\Delta(s,t)$ is defined as
\[
\Delta(s,t) := Q^+(s,t) +Q^-(s,t) + s \, a(t)  \left(Q^+(s,t)-Q^-(s,t)\right).
\]
with  $Q^{\sigma}(t,s):=Q(\psi(\sigma s,t))$ for $\sigma=\pm 1$ and $a(t):=s\langle \dot{\bar{\omega}}(t), \dot{\bar{\nu}}(t)\rangle/|\dot{\bar{\omega}}(t)|^2$. We claim that $\Delta(s,t) >0$ for all $(s,t)$ satisfying $t\in [0,+\infty)$ and $s\in N_t$.  On the one hand, we have
\[
a (t) = \frac{\langle \dot{\bar{\omega}}(t), \dot{\bar{\nu}}(t)\rangle}{|\dot{\bar{\omega}}(t)|^2} = \frac{-\bar{m}\left(\bar{m}-1\right) t^{\bar{m}-2}}{1+\bar{m}^2t^{m-2}} = - \bar{m}\left(\bar{m}-1\right) t^{\bar{m}-2} + o(t^{\bar{m}-2}).
\]
On the other hand, recalling that $Q^{\sigma}=4(\bar{\omega}_1+s\bar{\nu}_1)((\bar{\omega}_1+s\bar{\nu}_1)^2-(\bar{\omega}_2+s\bar{\nu}_2^m))$ (where we omit the dependence on $t$), the Newton binomial formula gives
\[
Q^{\sigma} = 4 \sum_{i=1}^{m+1} \sigma^{i} c_i s^{i} \quad  \mbox{with} \quad c_i= \binom{3}{i} \bar{\nu}_1^{i}\bar{\omega}_1^{3-i} - \binom{m}{i} \bar{\omega}_1 \bar{\nu}_2^{i}\bar{\omega}_2^{m-i} - \binom{m}{i-1} \bar{\nu}_1 \bar{\nu}_2^{i-1}\bar{\omega}_2^{m-i+1},
\]
with the convention that the binomial coefficient $\binom{k}{l}$ is zero when $l>k$. Thus, by setting $M:=(m+1)/2$, we obtain
\[
\Delta(s,t) = 8 \sum_{k=1}^{M} \left( c_{2k} + a c_{2k-1} \right) s^{2k} = 8 \sum_{k=1}^{2} \left( c_{2k} + a c_{2k-1} \right) s^{2k}  + 8 \sum_{k=3}^{M} \left( c_{2k} + a c_{2k-1} \right) s^{2k}.
\]
Expanding $c_i$ as a function of $t$ as $t \rightarrow 0$, we have
\[
c_i= \binom{3}{i} t^{\bar{m}(3-i)}\left( 1+o(1)\right) -  \binom{m}{i-1}  \left(-\bar{m}\right)^{i-1} t^{(\bar{m}-1)(i-1)} t^{m-i+1}\left( 1+o(1)\right),
\]
because the remaining terms are negligible. Note that for $i\geq 4$, the first term in the above expression vanishes. For $k\geq 3$, we verify that $c_{2k}+ac_{2k-1}>0$ for $t>0$ small, since $a<0$, $c_{2k}>0$ and $c_{2k-1}<0$. Moreover, using the Taylor expansion of $a(t)$ derived earlier and recalling that $m\geq 5$, we obtain
\[
c_2+ac_1  = 3t^{\bar{m}} + o (t^{\bar{m}}) \quad \mbox{and} \quad c_4+ac_3 = -\bar{m}\left( \bar{m}-1\right) t^{\bar{m}-2} + o (t^{\bar{m}-2}).
\]
In conclusion, for all $t,s$ with $t$ small, we have
\begin{eqnarray*}
\Delta(s,t) & \geq & 8 \left( 3t^{\bar{m}} + o (t^{\bar{m}})  \right) s^2 + 8 \left(-\bar{m}\left( \bar{m}-1\right) t^{\bar{m}-2} + o (t^{\bar{m}-2})\right) s^4\\
& = & 8 s^2 \left( 3t^{\bar{m}} - s^2 \bar{m}\left( \bar{m}-1\right) t^{\bar{m}-2} + o (t^{\bar{m}-2})\right).
\end{eqnarray*}
Since $\Psi(-s,t)\in \mathcal{E}$ for all $(s,t)$ with $t\in [0,+\infty)$ and $s\in N_t$, $\mathcal{E}$ is contained in the open half plane $\{x_1> 0\}$, and $\bar{\omega}_1(t)\geq0$ for all $t$ (by Lemma \ref{PROPPrelLEM1} (i)), it follows that $\bar{\omega}_1(t)-s\bar{\nu}_1(t)\geq 0$, which shows $s\leq t^{\bar{m}}$. Therefore, $\Delta(s,t)>0$  for all $(s,t)$ satisfying $t\in [0,+\infty)$ and $s\in N_t$, and as a consequence, by (\ref{22dec1}), we have $|A(\omega_R * \hat{\omega})| > |A(\omega|_{ [t_1,t_2]} * \hat{\omega})| $. To complete the argument, for each $r\in (0,1]$, define the Lipschitz curve $\omega_R^r:[t_1,t_2] \rightarrow \R^2$ by $\omega_R^r:=R_r\circ \omega|_{ [t_1,t_2]}$, which provides a continuous deformation of $\omega_R$. Since $|A(\omega_R * \hat{\omega})| > |A(\omega|_{ [t_1,t_2]} * \hat{\omega})| >0 $, where $A(\omega_R * \hat{\omega})$ and $A(\omega|_{ [t_1,t_2]} * \hat{\omega})$ have the same sign, and since $A(\omega_R^r * \hat{\omega})$ varies continuously with $r$ and tends to $0$ as $r \downarrow 0$, there is $\bar{r}>0$ such that
\[
A\left(\omega_R^r * \hat{\omega}\right)= A\left(\omega|_{ [t_1,t_2]} * \hat{\omega}\right).
\]
Consider the concatenation $\zeta:=\omega|_{ [0,t_1]}*\omega_R^{\bar{r}} *\omega|_{ [t_2,L(\omega)]}$. By the $1$-Lipschitz property of $R_{\bar{r}}$, we have $L(\omega_R^{\bar{r}})\leq L(\omega|_{ [t_1,t_2]})$. Moreover, by Lemma \ref{LEMStokesIso} (i), we can write
\begin{multline*}
 \int_{\zeta} P^2 \, dx_2  =   \int_{\omega|_{ [0,t_1]}} P^2 \, dx_2 + \int_{\omega_R^r} P^2 \, dx_2 + \int_{\omega|_{ [t_2,L(\omega)]}} P^2 \, dx_2 \\
  =  \int_{\omega} P^2 \, dx_2 + \int_{\omega_R^r} P^2 \, dx_2 - \int_{\omega} P^2 \, dx_2  = A(\omega_R^r * \hat{\omega}) - A\left(\omega|_{ [t_1,t_2]} * \hat{\omega}\right)= 0.
\end{multline*}
Thus, the curve $\zeta$ minimizes the length among all curves satisfying (\ref{OptimPlane}). However,  it is not analytic nor identical to $\bar{\omega}$, leading to a contradiction.

To prove (v), suppose that $\ell$ is a loop of $\omega$ and $t^*\in I_{\omega}\setminus (s_{\ell}^-,s_{\ell}^+)$ such that $\max_{t\in J_{\ell}}|Q(\omega(t))| \leq Q(\omega(t^*))$. Since the loop encloses a set with nonempty interior, we have $Q^*:=Q(\omega(t^*))>0$. For every $x_2\in \R$, the function $x_1 \geq 0 \mapsto Q(x_1,x_2)$ is convex. For $x_2\geq 0$ it vanishes at $x_1=x_2^{\bar{m}}$ with a nonnegative derivative, and for every $x_2< 0$, it vanishes at $x_1=0$ with a nonnegative derivative. Thus, the set $\{Q= Q^*,x_1\geq 0\}$ forms a curve, and by the implicit function theorem, there is a smooth function $\varphi: \R \rightarrow (0,+\infty)$ such that
\[
\Bigl\{(x_1,x_2) \in \R^2 \, \vert \, x_1 \geq 0, \, Q(x_1,x_2) = Q^*\Bigr\}=\Bigl\{ (\varphi(x_2),x_2) \, \vert \, x_2 \in \R \Bigr\}.
\]
Noting that $\varphi'(x_2)=-\frac{\partial Q}{\partial x_2}\left(\varphi(x_2),x_2\right)/\frac{\partial Q}{\partial x_1}\left(\varphi(x_2),x_2\right)$, we compute for any for every $x_2\in \R$:
\[
\varphi'(x_2) = \frac{m\varphi(x_2)x_2^{m-1}}{3\varphi(x_2)^2-x_2^m} \quad \mbox{and} \quad \varphi''(x_2) = \frac{m(m-1)\varphi(x_2)x_2^{m-2}  }{3\varphi(x_2)^2-x_2^m} -  \frac{2m^2 \varphi(x_2)x_2^{3m-2}  }{(3\varphi(x_2)^2-x_2^m)^3}.
\]
If $x_2\geq 0$, since $\varphi(x_2)> x_2^{\bar{m}}$ and $P(\varphi(x_2),x_2))=Q^*/(4\varphi(x_2))>0$, we obtain
\[
 \varphi''(x_2) \leq \frac{m(m-1)\varphi(x_2)x_2^{m-2}  }{2\varphi(x_2)^2+P(\varphi(x_2),x_2)} \leq \frac{m(m-1)}{2} \frac{x_2^{m-2}}{\varphi(x_2)} \leq  \bar{m}(m-1)  x_2^{\bar{m}-2}.
\]
If $x_2<0$, since for every $a>0$, the function $z\geq 0\mapsto za^3/(3z^2+a)^3$ attains its maximum at $z=\sqrt{a/15}$ with value $15^{5/2}\sqrt{a}/18^3$ and $\varphi(x_2)>0$, we have
\[
0 \leq \varphi''(x_2) \leq  -  \frac{2m^2 \varphi(x_2)x_2^{3m-2}  }{(3\varphi(x_2)^2-x_2^m)^3} = \frac{2m^2}{x_2^2} \frac{\varphi(x_2)(-x_2^{m})^3  }{(3\varphi(x_2)^2+ (-x_2)^m)^3} \leq \frac{2m^215^{\frac{5}{2}}}{18^3} |x_2|^{\bar{m}-2}.
\]
Therefore, there is a constant $K>0$ such that we have for every $\bar{x}_2\in [-1,1]$, we have
\[
\varphi \left(x_2\right) \leq \varphi \left(\bar{x}_2\right) +  \varphi'\left(\bar{x}_2\right) \left(x_2-\bar{x}_2\right) + K  \left(x_2-\bar{x}_2\right)^2 \qquad \forall x_2 \in \R.
\]
This shows that for sufficiently small $\epsilon>0$,  there is a disc $\mathcal{D}$ of radius $\rho:=L(\ell)/(2\pi)$, whose boundary is a circle passing through $\omega(t^*)$, and which is contained in $\{Q\geq Q(\omega(t^*))\}$. Therefore, let $\eta$ be a parametrization of the circle with length $L(\ell)$. By Lemma \ref{LEMStokesIso} (ii), we have
\[
|A(\ell)| \leq \frac{1}{4\pi} \sup_{x\in \mathcal{E}(\ell)} |Q(x)| \, L(\ell)^2 = Q(\omega(t^*)) \, \pi \rho^2  \leq \int_{\mathcal{D}} Q(x) \, dx = A(\eta),
\]
which contradicts obstruction (i).

To prove (vi), we assume that $\max_{t\in I_{\ell_1}}|Q(\omega(t))| \geq \max_{t\in I_{\ell_2}}|Q(\omega(t))|$ and consider $t^* \in I_{\ell_1}$ such that $Q(\omega(t^*)) = \max_{t\in I_{\ell_1}}|Q(\omega(t))|$. The result follows directly from (v).

To prove (vii) we proceed similarly to (v). Let $K>0$, $\ell$ a loop of $\omega$, and  $t^*\in I_{\omega}\setminus J_{\ell}$ be such that $\max_{t\in J_{\ell}} |Q(\omega(t))| \leq q^*$, with $q^*:=-Q(\omega(t^*))>0$, and $\omega_2(t^*)\geq K\epsilon$. For every $x_2>0$, the function $x_1 \geq 0 \mapsto Q(x_1,x_2)$ is convex, vanishes at $x_1=0$ and $x_1=x_2^{\bar{m}}$, and attains its minimum on the interval $[0,x_2^{\bar{m}}]$ at $x_1=h(x_2):=x_2^{\bar{m}}/\sqrt{3}$, with the value $-8x_2^{3\bar{m}}/(3\sqrt{3})$. By the implicit function theorem, there exist smooth functions $\varphi_{-}, \varphi_{+}: (x_2(q^*),+\infty) \rightarrow (0,+\infty)$, with $x_2(q^*)$ defined such that $-8x_2(q^*)^{3\bar{m}}/(3\sqrt{3})=-q^*$, satisfying $0<\varphi_{-} < h< \varphi_{+}<x_2^{\bar{m}}$, for which we have
\begin{multline*}
\Bigl\{\left(x_1,x_2\right) \in \R^2 \, \vert \, x_1 \geq 0, \, x_2>x_2\left(q^*\right), \, Q\left(x_1,x_2\right) = - q^*\Bigr\}\\
=\Bigl\{ \left(\varphi_-\left(x_2\right),x_2\right) \, \vert \, x_2>x_2\left(q^*\right) \Bigr\} \cup \Bigl\{ \left(\varphi_+\left(x_2\right),x_2\right) \, \vert \, x_2>x_2\left(q^*\right) \Bigr\}
\end{multline*}
and
\[
\Bigl\{x \in \R^2 \, \vert \, x_1 \geq 0, Q(x) < - q^*\Bigr\}=\Bigl\{ \left(x_1,x_2\right) \in \R^2 \, \vert \, \varphi_-\left(x_2\right)< x_1 < \varphi_+\left(x_2\right), \, x_2 > x_2\left(q^*\right) \Bigr\}.
\]
By Lemma \ref{PROPPrelLEM1} (iv) and (v), we have
\begin{eqnarray}\label{EQ30dec1}
q^* = 4\omega_1\left(t^*\right) \left|P\left(\omega\left(t^*\right)\right)\right| \leq 8 \epsilon^{\bar{m}} \beta \leq 8C \epsilon^{2m-1}.
\end{eqnarray}
Suppose $x^*=(x_1^*,x_2^*):=\omega(t^*)$ belongs to  the graph of $\varphi_-$. Then,
\[
P\left(x^*\right)= \left(x_1^*\right)^2 - \left(x_2^*\right)^m  =  \varphi_-\left(x_2^*\right)^2- \left(x_2^*\right)^m  <  h\left(x_2^*\right)^2 -\left(x_2^*\right)^m = - \frac{2\left(x_2^*\right)^m}{3}.
\]
Using the assumption $x_2^*\geq K \epsilon$ and $x^*_1 \geq C(K) \epsilon^{\bar{m}}$ (from Lemma \ref{PROPPrelLEM2}), it follows that $q^*\geq c\epsilon^{3\bar{m}}$ for some constant $c>0$, which contradicts (\ref{EQ30dec1}) for sufficiently small $\epsilon>0$. Consequently, we may assume that $x^*=\omega(t^*)$ lies on the graph of $\varphi_+$. As in the proof of (v), the first and second derivatives of $\varphi_+$ are given by
\[
\varphi_+'(x_2) = \frac{m\varphi_+(x_2)x_2^{m-1}}{3\varphi_+(x_2)^2-x_2^m} \quad \mbox{and} \quad \varphi_+''(x_2) = \frac{m(m-1)\varphi_+(x_2)x_2^{m-2}  }{3\varphi_+(x_2)^2-x_2^m} -  \frac{2m^2 \varphi_+(x_2)x_2^{3m-2}  }{(3\varphi_+(x_2)^2-x_2^m)^3}.
\]
The denominator $3(\varphi_+^2-h^2)$ is positive. Moreover, for every $x_2 \geq K\epsilon/2$, we have $\varphi_+(x_2) > h(x_2) \geq h(K\epsilon/2)$, implying, by (\ref{EQ30dec1}), that
\begin{eqnarray}\label{30dec2}
\varphi_+\left(x_2\right)^2- 2h^2\left(x_2\right) = \left (x^m_2-\frac{q^*}{4\varphi_+\left(x_2\right)}\right) - \frac{2x_2^m}{3} = \frac{x_2^m}{3} - \frac{q^*}{4\varphi_+(x_2)} >0,
\end{eqnarray}
for sufficiently small $\epsilon>0$. Consequently, for $x_2\in [K\epsilon/2,2\epsilon]$, we deduce
\[
\varphi_+''(x_2) \geq -  \frac{2m^2 \varphi_+(x_2)x_2^{3m-2}  }{(3\varphi_+(x_2)^2-x_2^m)^3} \geq -  \frac{2m^2x_2^{\bar{m}}  x_2^{3m-2}  }{(3h(x_2)^2)^3}  = - 2m^2 \, x_2^{\bar{m}-2}.
\]
Let $\mathcal{D}$ denote the open disc of radius $\rho:=L(\ell)/(2\pi)$, which touches the graph of $\varphi_+$ from below at $\omega(t^*)$. We claim that $\mathcal{D}$ lies within the set $\{|Q|\geq |Q(\omega(t^*))|=q^*\}$ for sufficiently small $\epsilon>0$. If not, there exists $x_2 \in [x_2^*-L(\ell),x_2^*+L(\ell)]$ such that $h(x_2)\geq \varphi_+(x_2^*)-L(\ell)$. Since $h$ is $1$-Lipschitz for small $\epsilon$, we infer $ \varphi_+(x^*_2)-h(x_2^*)\leq 2L(\ell)$. This contradicts the inequalities
\[
\varphi_+\left(x_2^*\right)-h\left(x_2^*\right)= \frac{\varphi_+\left(x_2^*\right)^2-h\left(x_2^*\right)^2}{\varphi_+\left(x_2^*\right)+h\left(x_2^*\right)} \geq \frac{h\left(x_2^*\right)^2}{\varphi_+\left(x_2^*\right)+h\left(x_2^*\right)} \geq 2 h\left(x_2^*\right) = \frac{2}{\sqrt{3}} \, \left(x_2^*\right)^{\bar{m}} \geq  \frac{2K^{\bar{m}}}{\sqrt{3}} \, \epsilon^{\bar{m}}
\]
and
\[
L(\ell)\leq C\beta^{1-\frac{1}{m}}\leq C^2\epsilon^{(3\bar{m}-1)(m-1)/m} = o\left(\epsilon^{\bar{m}}\right).
\]
following from (\ref{30dec2}) and Lemma \ref{PROPPrelLEM1} (v), (viii). We conclude as in (v).

To prove (viii), we assume that $\max_{t\in I_{\ell_1}}|Q(\omega(t))| \geq \max_{t\in I_{\ell_2}}|Q(\omega(t))|$. We consider $t^* \in I_{\ell_1}$ such that $|Q(\omega(t^*))| = \max_{t\in I_{\ell_1}}|Q(\omega(t))|$ and observe that, by Lemma \ref{PROPPrelLEM2}, we may assume that $\omega_2(t^*) \geq K \epsilon$. If $Q(\omega(t^*))\geq 0$, the result follows from (v); otherwise, it follows from (vii).
\end{proof}

\subsection{Intersections of $\omega$ with $\{P=0\}$ and loops }

By analyticity, the sets $\mbox{spt}(\omega)$ and $\{P=0\}$ intersect finitely many times. Consequently, we define $\tau_0=0 < \tau_1 < \cdots < \tau_N = L(\omega)$ such that
\[
\left(P\circ \omega \right)^{-1}(\{0\}) = \Bigl\{\tau_0=0, \ldots, \tau_N=L(\omega)\Bigr\}.
\]
We then set
\[
I_i := \bigl[\tau_i, \tau_{i+1}\bigr] \qquad \forall i \in \mathcal{I} := \left\{0, \ldots, N-1\right\}.
\]
The intervals $(\tau_i,\tau_{i+1})$ for $i\in \mathcal{I}$ are the maximal intervals where the function $t\mapsto P(\omega(t))$ is either strictly positive or strictly negative. By Lemma \ref{PROPPrelLEM1} (i), the same property holds for $t\mapsto Q(\omega(t))$. We denote by $\mathcal{I}^+$ (resp. $\mathcal{I}^-$) the set of $i\in \mathcal{I}$ such that $P\circ \omega|_{(\tau_i,\tau_{i+1})}>0$ (resp. $P\circ \omega|_{(\tau_i,\tau_{i+1})}<0$). For every $i\in \mathcal{I}^+$ (resp. $i \in \mathcal{I}^-$), we define $\sigma_i=1$ (resp. $\sigma_i=-1$). Finally, for each $i\in \mathcal{I}$,  we refer to the first loop of $\omega|_{ I_i}$, if it exists, as a loop $\ell$ associated with a subinterval $J_{\ell}=[s_{\ell}^-,s_{\ell}^+]\subset I_i$ such that $\omega|_{ [\tau_i,s_{\ell}^+)}$ is injective. Such a loop is necessarily simple.

The following result consolidates several properties essential for completing the proof of Proposition \ref{PROPMain}. Assertions (i)-(iv) are based on our preliminary results and the Gauss--Bonnet formula. Assertions (v), (vi) and (xi) are direct consequences of the obstructions described in Lemma \ref{LEMObs}. Lastly, assertions (vii)-(x) are established through a combination of curvature arguments, incorporating the Gauss--Bonnet formula, and the obstructions from Lemma \ref{LEMObs}.

\begin{lemma}\label{LEMAll0}
By taking $\epsilon_0>0$ from Lemma \ref{PROPPrelLEM1} smaller if necessary,  for any $\epsilon \in (0,\epsilon_0)$ and every $i\in \mathcal{I}$, then following properties hold.
\begin{itemize}
\item[(i)] The signed curvature $\kappa$ of the restriction $\omega|_{ I_i}$ satisfies $\sigma_i \lambda \kappa\geq 0$.
\item[(ii)] If $\omega|_{ I_i}$ admits a first loop $\ell$ then, by setting $\sigma:=\sigma_i \mbox{sgn}(\lambda)$, the set $\mathcal{E}(\ell)$ is strictly convex, $\mathcal{D}(\ell)=\mathcal{E}(\ell)=\mathcal{E}_{\sigma}(\ell)$,  $\sigma \mbox{ang}(\dot{\omega}(s_{\ell}^+),\dot{\omega}(s_{\ell}^-))\in (0,\pi)$, and $\int_{t\in J_{\ell}} \sigma \dot{\theta}(t)\, dt \in (\pi,2\pi)$.
\item[(iii)] $0\in \mathcal{I}^+$ and $\omega_2(\tau_{i+1})> \omega_2(\tau_0)=0$ for all $i\in \mathcal{I}$.
\item[(iv)] If $i\in \mathcal{I}^+$ and $i+1 \in \mathcal{I}$, then $i+1 \in \mathcal{I}^-$.
\item[(v)] If $i\in \mathcal{I}^-$, then $\omega|_{ I_i}$ admits a first loop $\ell$ associated with an interval $J_{\ell}\subset (\tau_i,\tau_{i+1})$.
\item[(vi)] There is at most one index $i \in \mathcal{I}^+$ such that $\omega|_{ I_i}$ is not injective.
\item[(vii)] If $i\in \mathcal{I}^+$, then $\omega|_{ I_i}$ admits at most one loop $\ell$.
\item[(viii)] If $\omega|_{ I_i}$ admits a first loop $\ell$, then for every $t\in [\tau_i,s_{\ell}^-]$, $\theta(t)\neq \sigma_i \pi/2 \, (\mbox{\rm mod } 2\pi)$.
\item[(ix)] If $\omega|_{ I_i}$ admits a unique loop, then $\lambda \cdot (\omega_2(\tau_{i+1})-\omega_2(\tau_{i})) \leq 0$.
\end{itemize}
Moreover, for every $K>0$, there exists $\epsilon_0(K)>0$ such that  for every $\epsilon \in (0,\epsilon_0(K))$ and every $i\in \mathcal{I}$, the following hold.
\begin{itemize}
\item[(x)] If $i\in \mathcal{I}^-$ and $\omega|_{ I_i}$ admits a first loop $\ell$ such that $\omega_2(s_{\ell}^-) \geq K\epsilon$ then $\ell$ is the unique loop of $\omega|_{ I_i}$.
\item[(xi)] There is at most one index $i \in \mathcal{I}^-$ such that $\omega|_{ I_i}$ admits a loop $\ell$ such that $\omega_2(s_{\ell}^-) \geq K \epsilon$.
\end{itemize}
\end{lemma}

\begin{proof}
Assertion (i) follows from the construction, (\ref{SYSomega}), and (\ref{GaussBonnet1}). To prove (ii), suppose that $\omega|_{ I_i}$ admits a first loop $\ell$ on an interval $[s_{\ell}^-, s_{\ell}^+]$. Reversing the orientation if necessary, we may assume that $\ell$ is positively oriented. Let $\delta$ be the discontinuity of the signed curvature of $\ell$ at $s_{\ell}^-$. Since $\dot{\ell}(s_{\ell}^-)\neq \dot{\ell}(s_{\ell}^+)$ (because $\omega$ is  solution of (\ref{SYSomega})), it follows that $\delta \neq 0$. If $\delta <0$, then the point $\eta(s_{\ell}^- -s)$ would lie in $\mathcal{D}(\ell)$ for sufficiently small $s>0$ (note that Lemma \ref{PROPPrelLEM1} (i)-(ii) ensure that $s_{\ell}^->0$). However, $\mathcal{D}(\ell)$ does not intersect the set $\{P=0\}$, which contains $\omega(0)=A_0$. Since $\omega|_{ [0,s_{\ell}^+)}$ is injective, this leads to a contradiction. We conclude the proof by applying Lemma \ref{LEMGaussBonnet1} (i), (\ref{GaussBonnet}), and (\ref{GaussBonnet1}), and by observing  that the signed curvature of $\ell$ is nonzero on its smooth part.

Assertion (iii) follows directly from Lemma \ref{PROPPrelLEM1} (i)-(ii). To prove (iv), suppose for contradiction that both $i$ and $i+1$ belongs to $\mathcal{I}^+$. In this case, the curve $t\mapsto \omega(t)$ must remain within the convex set $\{P\geq 0, \, x_1\geq 0\}$ for $t$ near $\tau_{i+1}$ and be tangent to the set $\{P=0\}$ at $\omega(\tau_{i+1})$. If $\omega(\tau_{i+1}))=0$, then we must have $\theta(\tau_{i+1})=\pm \pi/2  (\mbox{mod } 2\pi)$, which is prohibited (see the proof of  Lemma \ref{PROPPrelLEM1} (ii)). Otherwise, by Lemma \ref{PROPPrelLEM1} (i), $\omega(\tau_{i+1})$ belongs to $\{P\geq 0, \, x_1> 0\}$, a curve with nonzero curvature. It follows that $\dot{\theta}(\tau_{i+1})\neq 0$. However, because $P(\omega(\tau_{i+1}))=0$, equation (\ref{SYSomega}) implies that $\dot{\theta}(\tau_{i+1})=0$, leading to a contradiction.

Assertion (v) follows directly from obstruction (iv) in Lemma \ref{LEMObs}. Similarly, assertions (vi) and (xi) follow respectively from obstructions (vi) and (viii) in Lemma \ref{LEMObs}.

To prove (vii), we suppose for contradiction that $\omega|_{ I_i}$ admits a first loop $\ell$ on an interval $J_{\ell}=[s_{\ell}^-,s_{\ell}^+]\subset I_i$. By obstructions (vi)  of Lemma \ref{LEMObs}, the curve $\omega|_{ [s_{\ell}^+,\tau_{i+1}]}$ is injective. Therefore, it remains to show that the curves $\omega|_{ [\tau_i,s_{\ell}^+]}$ and $\omega|_{ (s_{\ell}^+,\tau_{i+1}]}$  do not intersect. Assume, for the sake of contradiction, that there is $\bar{t}$ in $(s_{\ell}^+,\tau_{i+1}]$ such that $\omega(\bar{t}) \in  \mbox{spt}(\omega|_{ [\tau_i,s_{\ell}^+]})$ and $\omega(t) \notin  \mbox{spt}(\omega|_{ [\tau_i,s_{\ell}^+]})$ for all $t \in (s_{\ell}^+,\bar{t})$. Note that, by analyticity, such a $\bar{t}$ must exist if  the curves $\omega|_{ [\tau_i,s_{\ell}^+]}$ and $\omega|_{ (s_{\ell}^+,\tau_{i+1}]}$  intersect. Now, consider $\bar{s}\in [\tau_i,s_{\ell}^+]$ such that $\omega(\bar{s})=\omega(\bar{t})$. Observe that $\bar{s}\neq s_{\ell}^-$, because otherwise $\omega(\bar{s})=\omega(s_{\ell}^-)=\omega(s_{\ell}^+)$, which contradicts the injectivity of $\omega|_{ [s_{\ell}^+,\tau_{i+1}]}$. We distinguish two cases: $\bar{s}< s_{\ell}^-$ and $\bar{s}>s_{\ell}^-$.

Case $\bar{s}< s_{\ell}^-$. Define $\bar{\ell} = \omega|_{[\bar{s},s_{\ell}^-]}*\omega|_{[s_{\ell}^+,\bar{t}]}$. By construction, $\bar{\ell}$ is a piecewise smooth continuous curve which is closed, simple, parametrized by arc length, and whose signed curvature has the same sign $\sigma =\pm 1$ as that of $\ell$ on its smooth segments. Since both $\ell$ and $\bar{\ell}$ have at most two singularities, Lemma \ref{LEMGaussBonnet1} (ii) gives $\mathcal{D}(\bar{\ell}) = \mathcal{E}_{\sigma}(\bar{\ell}) = \mathcal{E} (\bar{\ell})$ and $\mathcal{D}(\ell) = \mathcal{E}_{\sigma}(\ell) = \mathcal{E} (\ell)$. If $\sigma=1$, then both $\ell$ and $\bar{\ell}$ are positively oriented. By assertion (ii), $\mathcal{E}(\ell)$ is a convex set, and $\mbox{ang}(\dot{\omega}(s_{\ell}^+),\dot{\omega}(s_{\ell}^-))\in (0,\pi)\in (0,\pi)$. Consequently, for every nonzero vector $u$ such that $\mbox{ang}( \dot{\omega}(s_{\ell}^+),u)\in  (\mbox{ang}(\dot{\omega}(s_{\ell}^+),\dot{\omega}(s_{\ell}^-),\pi)$, we have  $\mbox{ang}( \dot{\omega}(s_{\ell}^-),u)\in  (\mbox{ang}(\dot{\omega}(s_{\ell}^-),\dot{\omega}(s_{\ell}^+),\pi)$. By Lemma \ref{LEM2025}, it follows that $\mathcal{E} (\ell)\subset \mathcal{E} (\bar{\ell})$. Therefore, there exists a translation of the loop $\ell$ in the directions $(0,\pm1)$ that intersects $\mbox{spt}(\bar{\ell})$. Lemma \ref{LEMObs} (iii) provides an obstruction. The case $\sigma=-1$ follows analogously.

Case $\bar{s}> s_{\ell}^-$. Define $\bar{\ell} = \omega|_{[\bar{s},\bar{t}]}$. Since $\omega|_{ [s_{\ell}^+,\tau_{i+1}]}$ is injective, $\bar{\ell}$  is a piecewise smooth continuous curve which is closed, simple, parametrized by arc length, and whose signed curvature has the same sign $\sigma =\pm 1$ as that of $\ell$ on its smooth segments. Noting that both $\ell$ and $\bar{\ell}$ have only one singularity, Lemma \ref{LEMGaussBonnet1} (ii) implies $\mbox{spt}(\omega|_{ (s_{\ell}^-,\bar{s})}) \subset \mathcal{E} (\bar{\ell})$. Consequently, there exists a translation of the loop $\ell$ in the directions $(0,\pm1)$ that intersects $\mbox{spt}(\bar{\ell})$. Lemma \ref{LEMObs} (iii) then provides an obstruction.

Assertion (x) follows in the same manner by noting that the assumption $\omega_2(s_{\ell}^-) \geq K \epsilon$, together with Lemma \ref{PROPPrelLEM2}  and  obstruction (viii) of Lemma \ref{LEMObs}, implies that the curve $\omega|_{ [s_{\ell}^+,\tau_{i+1}]}$ is injective.

To prove (viii), we assume for the sake of contradiction that there is $\bar{t}\in [\tau_i,s_{\ell}^-]$ such that $\theta(\bar{t})=\sigma_i \pi/2  (\mbox{mod } 2\pi)$. We address the case $\sigma_i=1$, with the other case left to the reader. Since the curve $\omega$ points toward the set $\{P\geq0\}$ at $\omega(\tau_i)$ and remains within this set on the interval $J_{\ell}$, we must have $\bar{t}> \tau_i$ and $P(\omega(\bar{t}))>0$. Let $\bar{h}> 0$ be the infimum of $h> 0$ such that $x(h):=\omega(\bar{t})+h(0,1)$ lies in $\mbox{spt}(\omega|_{ [\tau_i,\bar{t}]}) \cup \{P=0\}$. Note that $\bar{h}>0$ because $P(\omega(\bar{t}))>0$ and $\omega|_{[\tau_i,\bar{t}]}$ is injective. If $x(\bar{h})$ belongs to $\mbox{spt}(\omega|_{ [\tau_i,\bar{t}]})$, then consider the unique $\tau \in [\tau_i,\bar{t})$ such that $x(\bar{h})=\omega(\tau)$, and define the curve $\eta := \omega|_{ [\tau,\bar{t}]} * [\omega(\bar{t}),\omega(\tau)]$. Otherwise, define $\eta := \omega|_{ [\tau_i,\bar{t}]} * [\omega(\bar{t}),x(\bar{h})] * P_{[x(\bar{h}),\omega(\tau_i)]}$, where $P_{[x(\bar{h}),\omega(\tau_i)]}$ denotes the segment of the curve $\{P=0,\, x_1\geq 0\}$ connecting $x(\bar{h})$ to $\omega(\tau_i)$. In both cases, the curve $\eta$ is closed and simple. Since $\theta(\bar{t})=\sigma_i \pi/2  (\mbox{mod } 2\pi)$ and $\dot{\theta}(\bar{t})\neq 0$, we have $\omega(\bar{t}+s) \in \mathcal{D}(\eta)$ for $s>0$ small. Moreover, the restriction $\omega|_{[\tau_i,s_{\ell}^+)}$ is injective, does not intersect $\{P=0\}$ and, by Lemma  \ref{LEMGaussBonnet}, $\omega|_{ [\bar{t},s_{\ell}^+)}$ does not intersect the segment $[\omega(\bar{t}),x(\bar{h})]$. Consequently, the curve $\omega|_{ (\bar{t},s_{\ell}^+)}$, and hence the loop $\ell$, is contained in  $\mathcal{D}(\eta)$. Therefore, a translation of the loop $\ell$ in the directions $(0,-1)$ must intersect $\mbox{spt}(\eta)$, specifically $\mbox{spt}(\omega|_{[\tau_i,\bar{t}]})$.  This leads to a contradiction, as stated  in obstruction (iii) of Lemma \ref{LEMObs}.

To prove (ix), we assume that $\omega|_{ I_i}$ contains a unique loop $\ell$, and we define the closed, simple curve
\[
\eta := \omega|_{ [\tau_i,s_{\ell}^-]} * \omega|_{ [s_{\ell}^+,\tau_{i+1}]} *  P_{[\omega(\tau_{i+1}),\omega(\tau_i)]},
\]
where $P_{[\omega(\tau_{i+1}),\omega(\tau_i)]}$ denotes the segment of the curve $\{P=0,\, x_1\geq 0\}$ connecting $\omega(\tau_{i+1})$ to $\omega(\tau_i)$. We consider the case $i\in \mathcal{I}^+$ and $\lambda >0$. In this setting, by assertion (i), the signed curvature of $\omega|_{ [\tau_i,s_{\ell}^-]}$ and $ \omega|_{ [s_{\ell}^+,\tau_{i+1}]}$ is nonnegative. Suppose, for the sake of contradiction, that $\omega_2(\tau_{i+1})>\omega_2(\tau_i)$. In this case, the segment $P_{[\omega(\tau_{i+1}),\omega(\tau_i)]}$ also has nonnegative signed curvature. Thus, by Lemma \ref{LEMGaussBonnet1} (ii), we have $\mathcal{D}(\eta) = \mathcal{E}_{\sigma}(\eta) = \mathcal{E}(\eta)$. Referring to assertion (ii), we deduce that  for every nonzero vector $u$ such that $\mbox{ang}( \dot{\omega}(s_{\ell}^+),u)\in  (\mbox{ang}(\dot{\omega}(s_{\ell}^+),\dot{\omega}(s_{\ell}^-),\pi)$, we have $\mbox{ang}( \dot{\omega}(s_{\ell}^-),u)\in  (\mbox{ang}(\dot{\omega}(s_{\ell}^-),\dot{\omega}(s_{\ell}^+),\pi)$. By Lemma \ref{LEM2025}, we conclude that $\mathcal{E} (\ell)\subset \mathcal{E} (\bar{\ell})$. This contradicts obstruction (iii) of Lemma \ref{LEMObs}. The other cases can be proven following the same reasoning.
\end{proof}

\subsection{A closer look at the first loop of $\omega$}

The results established in Lemma \ref{LEMAll0} are not yet sufficient to prove Proposition \ref{PROPMain}. To complete the proof, we must verify the necessary conditions stated in assertions (x) and (xi). These conditions will be derived from an analysis of the first loop of $\omega$. By Lemma \ref{PROPPrelLEM1} (vi), the curve $\omega$ is not injective. Therefore, we consider its first loop, defined on an interval $J_0:=[s_0^-,s_0^+]$, where $\omega(s_{0}^-)=\omega(s_{0}^+)$, and $s_0^+\in I_{\omega}$ is the smallest $s\in I_{\omega}$ such that $\omega|_{ [0,s)}$ is injective but $\omega|_{ [0,s]}$ is not. We denote this loop by $\ell_0$ and  set  $L_0:=L(\ell_0)$. By obstruction (iv) of Lemma \ref{LEMObs}, $\omega|_{ I_1}$ is not injective. Consequently, either $\omega|_{ I_0}$ is not injective, in which case $J_0\subset I_0$, or $\omega|_{ I_0}$ is injective and $J_0\subset I_1$. In summary, we have
\begin{eqnarray}\label{EQFirst1}
J_0 \subset I_0 \quad \mbox{or} \quad J_0\subset I_1.
\end{eqnarray}
The objective of the present section is to prove the following result.

\begin{proposition}\label{PROPfirstloop}
There is $c>0$ such that, by taking $\epsilon_0>0$ from Lemma \ref{PROPPrelLEM1} smaller if necessary, for any $\epsilon \in (0,\epsilon_0)$, the following holds:
\begin{eqnarray}
\omega_1(s)  \geq c \, \epsilon^{\bar{m}} \quad \mbox{and}  \quad \omega_2(s)  \geq c \, \epsilon \qquad \forall s \in \left[s_0^-,L(\omega)\right].
\end{eqnarray}
\end{proposition}

The proof of Proposition \ref{PROPfirstloop} will follow from several lemmas. Before starting, we define $\beta_0$, $t_0\in I_0$ and $x_0>0$, $y_0\in \R$, and $\delta_0>0$ by
\[
\beta_0 := \max_{t\in J_0} \left|P(\omega(t))\right|= \left| P(\omega\left(t_0\right))\right| >0, \quad x_0:= \omega_1(t_0),\quad y_0:= \omega_2(t_0), \quad \delta_0:= \frac{\beta_0}{\max \left\{x_0,\left|y_0\right|^{\bar{m}}\right\}}.
\]
The first lemma is the following.

\begin{lemma}\label{PROPFirstloopLEM1}
For every $K > 0$, there are $\epsilon_0(K), c(K) >0$ such that there holds
\[
L_0 \geq K\delta_0 \quad \Longrightarrow \quad y_0 \geq c(K)\, \epsilon.
\]
\end{lemma}
\begin{proof}
Let $K>0$ be fixed. Consider the point $p:=(\epsilon^{\bar{m}}+L_0/4,\epsilon)$. We note that the length of the segment $[\omega(\epsilon),p]$ is equal to $L_0/4$ and
\[
Q(p) = 4 \left( \epsilon^{\bar{m}}+\frac{L_0}{4}\right) \left(\left( \epsilon^{\bar{m}}+\frac{L_0}{4}\right)^2 -\epsilon^m\right)  = 2L_0 \left( \epsilon^{\bar{m}}+\frac{L_0}{4}\right) \left( \epsilon^{\bar{m}}+\frac{L_0}{8}\right) \geq 2 \epsilon^m L_0.
\]
From the study of the level set $\{Q=Q(p),\, x_1\geq 0\}$ conducted in the proof of assertion (v) of Lemma \ref{LEMObs},  there exists a disc $D\subset \R^2$ such that $Q(x) \geq Q(p)$ for all $x \in D$, and whose boundary is a circle $\nu$ such that $\nu(0)=p$ and $L(\nu)=L_0/2$. Let $\eta$ be the curve given by the concatenation of $[\bar{\omega}(\epsilon),p]$, $\nu$, and $[p,\bar{\omega}(\epsilon)]$. By construction $L(\eta)= L_0$, so by Lemma  \ref{LEMObs} (i), we have necessarily, since $\omega$ is optimal,
\begin{eqnarray}\label{12dec1}
|A(\ell_0)| > |A(\eta)|.
\end{eqnarray}
Assume that $L_0 \geq K\delta_0$. We estimate $|A(\ell_0)|$ from above and  $|A(\eta)|$ from below. By Lemma \ref{LEMStokesIso} (ii), we have
\[
|A(\ell_0)| \leq  \frac{1}{\pi} \sup_{x\in \mathcal{E}(\ell_0)} |x_1P(x)| \, L_0^2 \leq \frac{\beta_0 L_0^2}{\pi}  \sup_{x\in \mathcal{E}(\ell_0)} x_1 \leq \frac{\left(x_0+L_0\right)\beta_0L_0^2}{\pi},
\]
where we used that $\sup_{x_1\in \mathcal{E}(\ell_0)}(x_1)=\max_{t\in I_0} \omega_1(t) \leq x_0 + L_0$, since $|\dot{\omega}_1|\leq 1$. Moreover, using the above lower bound for $Q(p)$, $\mathcal{L}(D)=L_0^2/(16\pi)$, and the assumption, we have
\[
|A(\eta)| = | \mathcal{L}_Q\left(\mathcal{E}(\eta)\right)| \geq \mathcal{L}(\mathcal{E}(\eta)) \, Q(p) \geq \frac{ \epsilon^m L_0^3}{8\pi}  \geq \frac{ \epsilon^m L_0^2K\delta_0}{8\pi} = \frac{ \epsilon^m L_0^2K\beta_0}{8\pi \max \left\{x_0,\left|y_0|^{\bar{m}}\right|\right\}}.
\]
Plugging the bounds on  $|A(\ell_0)|$ and  $|A(\eta)|$ in (\ref{12dec1}), we obtain
\[
(x_0+L_0) \max \left\{x_0+L_0,|y_0|^{\bar{m}}\right\} \geq (x_0+L_0) \max \left\{x_0,|y_0|^{\bar{m}}\right\} \geq \frac{K}{8} \, \epsilon^m.
\]
If $x_0+L_0 \leq |y_0|^{\bar{m}}$, the above inequality implies $|y_0^m|\geq K \epsilon^m/8$, which proves the result. Alternatively, if $x_0+L_0>|y_0|^{\bar{m}}$,  the inequality gives $x_0+L_0 \geq c'  \epsilon^{\bar{m}}$ with $c':=\sqrt{K/8}$. By Lemma \ref{PROPPrelLEM1} (viii) and (v), we know that $L_0 \leq C\beta^{1-\frac{1}{m}}\leq C^2 \epsilon^{(3\bar{m}-1)(1-\frac{1}{m})}$, where $(3\bar{m}-1)(1-\frac{1}{m})>\bar{m}$ (because $m\geq 5$). Thus, for sufficiently small $\epsilon>0$, we have $x_0\geq c'\epsilon^{\bar{m}}/2$. Since $\beta_0\leq \beta=o(\epsilon^m)$ (by Lemma \ref{PROPPrelLEM1} (v)), it follows that
\[
y_0^m = x_0^2 - \beta_0 \geq \frac{c'}{2} \, \epsilon^m + o \left(\epsilon^m\right),
\]
which concludes the proof.
\end{proof}

Our objective is now to demonstrate that the assumption of  Lemma \ref{PROPFirstloopLEM1} holds for some constant $K>0$ depending only on $m$. To this end, we begin with the following preparatory lemma.

\begin{lemma}\label{PROPFirstloopLEM2}
There is $c>0$ such that, by taking $\epsilon_0>0$ from Lemma \ref{PROPPrelLEM1} smaller if necessary, for any $\epsilon \in (0,\epsilon_0)$, the following holds:
\begin{eqnarray}\label{3janv1}
x_0 \geq c \, \beta_0^{\frac{1}{2}} \quad \mbox{and} \quad y_0 \geq - \beta_0^{\frac{1}{2}}.
\end{eqnarray}
\end{lemma}

 \begin{proof}
We consider separately the cases where $J_0\subset I_0$ and  $J_0\subset I_1$ (see (\ref{EQFirst1})).  \\

\noindent Case $J_0\subset I_0$: Since $\beta_0=P(\omega(t_0))>0$ in this case, we have $x_0^2=\beta_0+y_0^m$, which shows that the inequality for $x_0$ in (\ref{3janv1}) follows from the inequality for $y_0$, with $c=1/2$ and for sufficiently small $\epsilon>0$. To prove that $y_0 \geq -\beta_0^{1/2}$, we proceed by contradiction and assume $y_0 < -\beta_0^{1/2}$. If $\omega_2(t)\geq 0$ for some $t\in J_0$ then $L_0 \geq 2|y_0| \geq 2 \beta_0^{1/2}$ and thus (using that $\beta_0=x_0^2-y_0^m= x_0^2+|y_0|^m \leq 2 \max \{x_0^2,|y_0|^m\}$)
\[
L_0 \geq 2 \beta_0^{\frac{1}{2}} = 2 \beta_0 \beta_0^{-\frac{1}{2}} \geq \sqrt{2} \frac{|P(\omega(t_0))|}{\max \{x_0,|y_0|^{\bar{m}}\}} = \sqrt{2} \, \delta_0,
\]
which, by Lemma \ref{PROPFirstloopLEM1}, yields $y_0\geq c(\sqrt{2})\epsilon$, a contradiction. Therefore, we must have $\omega_2(t)<0$ for all $t\in J_0$. This implies $\omega_1(t)<P(\omega(t))^{1/2}\leq \beta_0^{1/2}$ for all $t\in J_0$. Let $s_1, s_2 \in I_{\omega}$ be the maximal interval containing $t_0$ such that $\omega_2(t)<0$ for all $t\in (s_1,s_2)$.  By obstruction (v) of Lemma \ref{LEMObs}, we infer that
\[
\max_{t\in J_0}Q(\omega(t))  = \max_{t\in J_0}|Q(\omega(t))| > \max_{t\in I_{\omega} \setminus J_0}|Q(\omega(t))| \geq  \max_{t\in (s_1,s_2) \setminus J_0} Q(\omega(t)),
\]
which implies
\begin{eqnarray}\label{EQFirst3}
Q\left(\omega \left(s_1\right)\right), Q\left(\omega\left(s_2\right)\right) \leq \max_{t\in (s_1,s_2) } Q(\omega(t)) = \max_{t\in J_0} Q(\omega(t)) \leq 4 \beta_0 \, \max_{t\in J_0} \omega_1(t) \leq 4 \beta_0^{\frac{3}{2}}.
\end{eqnarray}
Consequently, we have $\omega_1(s_1), \omega_1(s_2) \leq \beta_0^{1/2}$, and thus
\[
L\left(\left[\omega_1\left(s_1\right), \omega_1\left(s_2\right)\right]\right) = \left|\omega_1\left(s_1\right)- \omega_1\left(s_2\right)\right |\leq \beta_0^{1/2}.
\]
Since $L(\omega|_{ (s_1,s_2)})\geq 2|y_0|\geq 2 \beta_0^{1/2}$, it follows that
\begin{eqnarray}\label{3janv2}
\Lambda := \frac{ L \left(\omega|_{ (s_1,s_2)}\right) - L\left([\omega_1(s_1), \omega_1(s_2)]\right) }{4} \geq \frac{\beta_0^{\frac{1}{2}}}{4}.
\end{eqnarray}
Set $p=\bar{\omega}(\epsilon)+(\Lambda,0)$ and consider the curve $\eta$ formed by the concatenation of $[\bar{\omega}(\epsilon),p]$, $\partial D$, and $[p,\bar{\omega}(\epsilon)]$, where $D$ is a closed disc containing $p$ on its boundary such that $\min_{x\in D} Q(x) =Q(p)$ and $L(\partial D)=2\Lambda$. We will conclude by obstruction (ii) of Lemma \ref{LEMObs}. Let $\alpha:= \omega|_{ [s_1,s_2]} * [\omega(s_2),\omega(s_1)]$. Observe that
\[
L(\alpha) = 4 \Lambda + 2 L([\omega(s_2),\omega(s_1)]) \leq 4 \Lambda + 2 \beta_0^{\frac{1}{2}} \leq 12 \Lambda,
\]
where in the last inequality we used (\ref{3janv2}). By Lemma \ref{LEMStokesIso} (ii), and using (\ref{EQFirst3})), we deduce that
\[
|A(\alpha)| \leq \frac{L(\alpha)^2}{4\pi} \max_{x\in \mathcal{E}(\alpha)}|Q(x)| \leq 144 \Lambda^2  \max_{t\in (s_1,s_2)}|Q(\omega(t))| \leq  144 \Lambda^2 \beta_0^{\frac{3}{2}}.
\]
On the other hand, by construction, we have $|A(\eta)|\geq \Lambda^2Q(p)/\pi$, where
\[
Q(p) = 4 \left( \epsilon^{\bar{m}}+\Lambda\right) \left( \left( \epsilon^{\bar{m}}+\Lambda\right)^2 - \epsilon^m\right) = 4 \Lambda \left( \epsilon^{\bar{m}}+\Lambda\right)\left( 2\epsilon^{\bar{m}}+\Lambda\right) \geq 8 \Lambda \epsilon^m \geq 2 \epsilon^m \beta_0^{\frac{1}{2}}.
\]
This shows that $|A(\alpha)| \leq |A(\eta)|$ for sufficiently small $\epsilon>0$ (since $\beta_0\leq \beta =o(\epsilon^m)$ by Lemma \ref{PROPPrelLEM1} (v)). Lemma \ref{LEMObs} (iv) provides an obstruction, establishing that $y_0 \geq -\beta_0^{1/2}$.\\

\noindent Case $J_0 \subset I_1$: Since $m$ is odd and $P(\omega(t_0))<0$, it follows that $y_0>0\geq - \beta_0^{\frac{1}{2}}$. We now aim to prove that $x_0 \geq c\beta_0^{1/2}$ for some constant $c>0$, provided that $\epsilon>0$ is sufficiently small.  In the following, for any $v, w \in \R^2 \setminus \{0\}$, let $\mathfrak{C}^+[v,w]$ denote the closed positive cone defined as the convex hull of the two half-lines in the directions $v$ and $w$. Next, we define $c_0:=\sqrt{3}/3$ and, for $c \in \left(0,c_0\right)$,
\[
U_c := \Bigl\{ \left(x_1,x_2\right) \in \R^2 \, \vert \, 0<x_2 < 3\epsilon, \, 0 < x_1 <cx_2^{\bar{m}}\Bigr\}.
\]
The proof proceeds in three steps. We  set $u:=(1,1)$, $v:=(0,1)$, and $w:=(1,-1)$.

Step 1: We claim that for all $c\in (0,c_0)$, there exists $\epsilon_c>0$ such that for any $\epsilon\in (0,\epsilon_c)$,
\begin{eqnarray}\label{20dec1}
 \langle \nabla Q(x), z\rangle <0 \quad \forall x \in U_c, \, \forall z \in \mathfrak{C}^+[v,w] \setminus \{0\}.
\end{eqnarray}
To verify this, note that $\nabla Q(x)=4(3x_1^2-x_2^m,-mx_1x_2^{m-1})$. First, the inequality $\langle \nabla Q(x),v\rangle<0$ follows directly. Additionally, we have $\langle \nabla Q(x),w\rangle = 4(3x_1^2+mx_1x_2^{m-1}-x_2^m)$. For $x_1\in (0,cx_2^{\bar{m}})$, this yields $\langle \nabla Q(x),w\rangle \leq 4x_2^m(3c^2+cmx_2^{\bar{m}-1}-1)$, which is negative for $c<c_0$ and  sufficiently small $x_2>0$. We conclude by linearity.

Step 2: We claim that for all $c\in (0,c_0)$, there exists $\epsilon_c>0$ such that for any $\epsilon\in (0,\epsilon_c)$, $ \mbox{spt}(\ell_0) \not\subset U_c$.
To prove this, we proceed by contradiction, assuming that $\mbox{spt}(\ell_0) \subset U_c$ for some $c\in (0,c_0)$. Since $J_0\subset I_1$, obstruction (iv) of Lemma \ref{LEMObs}, together with Lemma \ref{PROPPrelLEM1} (i), ensures that $Q\circ \ell_0<0$. We then define the following affine cones:
\[
\mathfrak{C}_1 := \omega(s_0^-) + \mathfrak{C}^+[u,-w], \quad \mathfrak{C}_2 := \omega(s_0^-)+\mathfrak{C}^+[-v,-w], \quad \mathfrak{C}_3 := \omega(s_0^-)+ \mathfrak{C}^+[-v,u].
\]
Since the union of these cones is equal to $\R^2$, at least one of the following cases must hold:
\[
\mbox{spt}(\ell_0) \subset \mathfrak{C}_1, \quad \left(\mbox{spt}(\ell_0) \setminus \{\omega(s_0^-)\} \right) \cap \mathfrak{C}_2 \neq \emptyset  \quad \mbox{or} \quad
\left\{ \begin{array}{l}
\left(\mbox{spt}(\ell_0) \setminus \{\omega(s_0^-)\}\right) \cap \mathfrak{C}_2 = \emptyset\\
\left(\mbox{spt}(\ell_0) \setminus \{\omega(s_0^-)\}\right) \cap \mathfrak{C}_3 \neq \emptyset.
\end{array}
\right.
\]
Our goal is to derive a contradiction in each of these cases.

In the case where $(\mbox{spt}(\ell_0) \setminus \{\omega(s_0^-)\}) \cap \mathfrak{C}_2 \neq \emptyset$, choose $t\in (s_0^-,s_0^+)$ such that $\omega(t)\in \mathfrak{C}_2$. Define $T^r:\R^2 \rightarrow \R^2$ as the translation by the vector $r\bar{v}$, where $\bar{v}:=\omega(s_0^-)-\omega(t)$ and $r>0$, and set $\eta^r:=T^r\circ \ell_0$. By Lemma \ref{LEMAll0} (ii), the support of $\ell_0$ encloses a strictly convex set that is contained in the cone $\omega(s_0^-)+\mathfrak{C}^+[\dot{\omega}(s_0^-),-\dot{\omega}(s_0^+)]$. Thus, the vector $-\bar{v}$ lies the interior of the cone $\mathfrak{C}^+[\dot{\omega}(s_0^-),-\dot{\omega}(s_0^+)]$. This shows that $\eta^r$ intersects $\omega|_{ [0,s_{0}^-]} * \omega|_{ [s_{0}^{+},L(\omega)]}$ for $r>0$ sufficiently  small. Since $T^r$ is an isometry, we have $L(\eta^r)=L(\ell_0)$. Furthermore, by (\ref{20dec1}), $Q\circ \ell_0<0$,  and since $\bar{v}\in \mathfrak{C}^+[v,w]$ (because $\omega(t) \in \mathfrak{C}_2$), we have $|A(\eta^r)|>|A(\ell_0)|$  for $r>0$ sufficiently  small. This violates obstruction (i) in Lemma \ref{LEMObs}, leading to a contradiction.

In the case where $\mbox{spt}(\ell_0) \subset \mathfrak{C}_1$, consider the rotation $T_{\phi}:\R^2 \rightarrow \R^2$ with angle $\phi$ around the point $\omega(s_0^-)$, and define $\eta_{\phi}:=T_{\phi}\circ \ell_0$. By construction, since $T_{\phi}$ is an isometry fixing $\omega(s_0^-)$, the set $\mbox{spt}(\eta_{\phi})$ intersects the support of $\omega|_{ [0,s_{0}^-]} * \omega|_{ [s_{0}^{+},L(\omega)]}$, and $L(\eta_{\phi})\leq L(\ell_0)$. Next, recalling that $\mbox{spt}(\ell_0)$ encloses a convex set, for any $x\in\mathcal{E}(\ell_0)$, the vector $x-\omega(s_0^-)$ belongs to $\mathfrak{C}^+[-w,u]$. Therefore, we have
\[
-(x-\omega(s_{0}^-))^\perp \in \mathfrak{C}^+[-w,u] \subset \mathfrak{C}^+[v,w] \qquad \forall x \in \mathcal{E}(\ell_0),
\]
where $(v_1,v_2)^\perp=(-v_2,v_1)$. Thus, by (\ref{20dec1}), it follows that $ \langle \nabla Q(x), -(x-\omega(s_{0}^-)^\perp\rangle<0$ for all $x\in\mathcal{E}(\ell_0)$. Since $\frac{\partial T_{\phi}}{\partial \phi}(0,\cdot)$ is the rotation of angle $\pi/2$ at the origin, one then deduces that
\begin{eqnarray*}
\left. \frac{d}{d\phi}\right|_{\phi=0} \int \int _{\mathcal{E}(T_\phi\circ \ell_0)} Q(x)dx_1 dx_2  & = & \left. \frac{d}{d\phi}\right|_{\phi=0}  \int \int _{\mathcal{E}(T_\phi\circ \ell_0)} Q(T_\phi(x))dx_1 dx_2\\
		& = &  \int \int _{\mathcal{E}(T_\phi\circ \ell_0)}  \langle \nabla Q(x), (x-\omega(s_{\ell_0}^-))^\perp\rangle dx_1 dx_2 >0.
\end{eqnarray*}
This implies that $|A(\eta^r)|>|A(\ell_0)|$ for small $\phi< 0$. Obstruction (i) in Lemma \ref{LEMObs} leads again to a contradiction.

We now address the case where $(\mbox{spt}(\ell_0) \setminus \{\omega(s_0^-)\}) \cap \mathfrak{C}_2 = \emptyset$ and $(\mbox{spt}(\ell_0) \setminus \{\omega(s_0^-)\}) \cap \mathfrak{C}_3 \neq \emptyset$. Since $\dot{\omega}(\tau_1)$ points toward the set $\{P<0\}$ at $\omega(\tau_1)$, we must have $\theta(\tau_1)\in (\pi/4,3\pi/2)+2k\pi$ for some $k\in \Z$. Without loss of generality, we assume $k=0$. If $\lambda <0$, then, by (\ref{SYSomega}), $\theta$ is increasing on $I_1$, thus by Lemma \ref{LEMAll0} (viii), we have $\theta(s_{0}^-)\in (\dot{\theta}(\tau_1),3\pi/2)\subset (\pi/4,3\pi/2)$. Since the support of $\ell_0$ encloses a convex set contained within the cone $\omega(s_0^-)+\mathfrak{C}^+[\dot{\omega}(s_0^-),-\dot{\omega}(s_0^+)]$ with $\mbox{ang}(\dot{\omega}(s_{\ell}^+),\dot{\omega}(s_{\ell}^-))\in (0,\pi)$ and $(\mbox{spt}(\ell_0) \setminus \{\omega(s_0^-)\}) \cap \mathfrak{C}_2 = \emptyset$, we deduce that the set $\mbox{spt}(\ell_0) \setminus \{\omega(s_0^-)\}$ does not intersect the affine cone $\mathfrak{C}_3$, leading to a contradiction. Therefore,  since $\lambda \neq 0$ by Lemma \ref{PROPPrelLEM1} (vii), we must have $\lambda >0$. By (\ref{SYSomega}), $\theta$ is decreasing on $I_1$. We claim that there exists a sequence $(\ell_k)_{k\in \N}$ of simple loops of $\omega|_{ I_1}$, associated with a sequence of intervals $(J_k=[s_k^-,s_k^+])_{k\in \N}$, such that $(s_k)_{k\in \N}$ is increasing and the following properties hold for each $k\in \N$:
\begin{eqnarray}\label{RECk}
\left\{
\begin{array}{l}
(a) \, \, \exists t \in \mbox{int}\left(J_k\right) \mbox{ such that } \omega(t) \in \omega \left(s_k^-\right) + \mathfrak{C}^+[-v,u],\\
(b) \, \, \forall s \in \mbox{int}\left(J_k\right), \, \omega(s) \in U_c \setminus \left( \omega \left(s_k^-\right) + \mathfrak{C}^+[-v,-w]\right).
\end{array}
\right.
\end{eqnarray}
This claim contradicts the analyticity of $\omega$. The base case for $k=0$ is satisfied by $\ell_0$, given our assumptions: $\mbox{spt}(\ell_0) \subset U_c$, $(\mbox{spt}(\ell_0) \setminus \{\omega(s_0^-)\}) \cap \mathfrak{C}_2 = \emptyset$, and $(\mbox{spt}(\ell_0) \setminus \{\omega(s_0^-)\}) \cap \mathfrak{C}_3 \neq \emptyset$. Now assume that there exists a loop $\ell_k$ for some $k\in \N$ satisfying (\ref{RECk}). We start by proving that
\begin{eqnarray}\label{3janv33}
\theta\left(s_k^+\right) \in \left( \frac{\pi}{2},\frac{5\pi}{4}\right] \, (\mbox{\rm mod } 2\pi).
\end{eqnarray}
We prove (\ref{3janv33}) by contradiction. Without loss of generality, we assume that $\theta (s_k^+)$ lies in $(-\pi,\pi]$. We then consider separately the three cases where $\theta(s_k^+)\in (-\pi/4,\pi/2)$, $\theta(s_k^+)\in (-3\pi/4,-\pi/4]$, and $\theta(s_k^+)=\pi/2$.
\begin{itemize}
\item If  $\theta(s_k^+)\in (-\pi/4,\pi/2)$, then $\omega(s_k^+-s) \in \omega(s_k^-) + \mathfrak{C}^+[-v,-w]$ for sufficiently small $s>0$, which contradicts property  (b) in (\ref{RECk}).
\item If $\theta(s_k^+) \in (-3\pi/4,-\pi/4]$, then the set $\mbox{spt}(\ell_k)\setminus \{\omega(s_k^-)\}$ is contained within the interior of the affine cone $\omega(s_k^-)+  \mathfrak{C}^+[-w,-\dot{\omega}(s_k^+)]$. This follows from property (b) satisfied by $\ell_k$, the strict convexity of the region enclosed by its support (see Lemma \ref{LEMAll0} (ii)), and the fact that $\lambda >0$, which implies that $\ell_k$ is negatively oriented. Consequently, $\mbox{spt}(\ell_k)\setminus \{\omega(s_k^-)\}$ is entirely contained within the interior of the affine cone $\omega(s_k^-) + \mathfrak{C}^+[-w,u]$, contradicting property (a)  in (\ref{RECk}).
\item If $\theta(s_k^+)=\pi/2$, then since $\theta$ is monotone decreasing, a translation of $\mbox{spt}(\ell_k)$ in the direction $(0,1)$ intersects the support of $\omega|_{ (s_k^+,\tau_2]}$. This contradicts obstruction (ii) in Lemma \ref{LEMObs}.
\end{itemize}
Therefore, the proof of (\ref{3janv33}) is complete, and as a consequence, we have
\[
\omega\left( s_k^+ + s\right) \in  \omega\left( s_k^+\right) + \mathfrak{C}^+\left[ \dot{\omega}\left(s_k^+\right), v\right] \subset U_c \qquad \forall s>0 \mbox{ small}.
\]
By obstruction (iii) of Lemma \ref{LEMObs}, we have $\mbox{spt}(\omega_{(s_k^+,\tau_2]})\cap \{\omega(s_k^+)+t(0,1)\, \vert \, t\geq 0\}=\emptyset$. Consequently, since $\omega(\tau_2) \notin \omega (s_k^k) + \mathfrak{C}^+[\dot{\omega}(s_k^+), v]$, the intersection $\mbox{spt}(\omega_{(s_k^+,\tau_2]})\cap \{\omega(s_k^+)+t\dot{\omega}(s_k^+)\, \vert \, t\geq 0\}$ must be nonempty. By Lemma \ref{LEMGaussBonnet}, it follows that $\omega_{(s_k^+,\tau_2]}$ contains a loop within $U_c$. We define $\ell_{k+1}$ as the first simple loop of $\omega_{(s_k^+,\tau_2]}$. The fact that $\ell_{k+1}$ satisfies (\ref{RECk}) can be demonstrated using the same arguments employed to show that $(\mbox{spt}(\ell_0) \setminus \{\omega(s_0^-)\}) \cap \mathfrak{C}_2 \neq \emptyset$ and $\mbox{spt}(\ell_0) \subset \mathfrak{C}_1$ are not possible. This completes the proof of Step 2.

Step 3: We complete the proof of Lemma \ref{PROPFirstloopLEM2} by proving that $x_0 \geq \beta_0^{1/2}/2$, provided that $\epsilon>0$ is sufficiently small.  Consider the curves
\[
\Upsilon := \Bigl\{ x \in \R^2 \, \vert  \, 2x_1=x_2^{\bar{m}}, \, x_2 \geq 0 \Bigr\} \quad \mbox{and} \quad \Xi_{\rho} :=  \Bigl\{ x \in \R^2 \, \vert  \, P(x)=-\rho,  \, 2x_1 \in \left[0, \rho^{\frac{1}{2}}\right] \Bigr\},
\]
for $\rho>0$. The intersection $\Upsilon \cap \Xi_1$ is empty because a point $(x_1,x_2)$ in the intersection must satisfy $4x_1^2=x_2^m=x_1^2+1$, which implies $x_1=\sqrt{3}/3 > 1/2$. Moreover, for every $\rho\in (0,1]$, the map $\Delta_{\rho}:\R^2 \rightarrow \R^2$ defined by $\Delta_{\rho}(x):=(\rho^{-1/2}x_1,\rho^{-1/m}x_2)$ is $\rho^{-1/2}$-Lipschitz, and it satisfies $\Upsilon=\Delta_{\rho}(\Upsilon)$ and $\Xi_{\xi,1}=\Delta_{\rho}(\Xi_{\xi,\rho})$. Hence, denoting for every $\rho >0$ by $\mbox{dist}(\Upsilon,\Xi_{\rho})$ the infimum of $|x-x'|$ for $x\in \Upsilon$ and $x'\in \Xi_{\rho}$, we have
\begin{eqnarray}\label{4janv1}
\mu:=\mbox{dist}\left(\Upsilon,\Xi_{1}\right)  >0 \quad \mbox{and} \quad \mbox{dist}\left(\Upsilon,\Xi_{\rho}\right) \geq \mu \rho^{\frac{1}{2}} \qquad \forall \rho \in (0,1].
\end{eqnarray}
Set $c:=1/2 \in (0,c_0)$ and suppose, by contradiction, that $x_0< \sqrt{\beta_0}/2$. Then, $\omega(t_0)=(x_0,y_0) \in \mbox{spt}(\ell)\cap \Xi_{\beta_0} \cap U_{c}$ (because $\beta_0=y_0^m-x_0^2>4x_0^2$). However, $\Upsilon$ contains the boundary of $U_{c}$ within the open set $\{x_1>0\}$ and, by Step 2, $\mbox{spt}(\ell_0)$ is not entirely contained in $U_{c}$ for $\epsilon\in (0,\epsilon_c)$. Therefore, we may assume that $\mbox{spt}(\ell_0)\cap \Upsilon\neq \emptyset$. From (\ref{4janv1}), we deduce that
\[
L_0 \geq \mbox{dist} \left(\Upsilon,\Xi_{\beta_0}\right) \geq \mu \beta_0^{\frac{1}{2}} \geq  \mu \frac{\beta_0}{y_0^{\bar{m}}} =  \mu \frac{\beta_0}{\max \left\{x_0,\left|y_0\right|^{\bar{m}}\right\}} = \mu \delta_0,
\]
where we used the relations $\beta_0=-P(\omega(t_0))=y_0^m-x_0^2\leq y_0^m$ and $y_0^m\geq x_0^2$. By Lemma \ref{PROPFirstloopLEM1}, the inequality above implies $y_0 \geq c(\mu)\epsilon$. Using $x_0<\sqrt{\beta_0}/2$ and Lemma \ref{PROPPrelLEM1} (v), we finally obtain
\[
c(\mu) \epsilon \leq y_0  =\left(x_0^2+\beta_0\right)^{\frac{1}{m}} \leq \left(\frac{5}{4}\right)^{\frac{1}{m}} \beta_0^{\frac{1}{m}} \leq \left(\frac{5C}{4}\right)^{\frac{1}{m}} \epsilon^{\frac{3}{2}-\frac{1}{m}},
\]
which leads to a contradiction for sufficiently small $\epsilon>0$, as $m\geq 5$. This completes the proof of Lemma \ref{PROPFirstloopLEM2}.
\end{proof}

Our next lemma is the following.

\begin{lemma}\label{LEM3janvier}
By taking $\epsilon_0>0$ smaller and $C>0$ larger, if necessary, as prescribed in Lemma \ref{PROPPrelLEM1}, it follows that for any $\epsilon \in (0,\epsilon_0)$, the inequality $|\lambda| \beta_0^2 \leq C$ holds.
\end{lemma}

\begin{proof}
Set $M:=\max \{|y_0|^{\bar{m}},\sqrt{\beta_0}\}$. There exists $\sigma=\pm 1$ such that $\beta_0=\sigma (x_0^2-y_0^m)>0$. Therefore, we have $x_0^2=\sigma\beta_0 + y_0^m \leq \beta_0 + |y_0|^m\leq 2M^2$. If $M=\sqrt{\beta_0}$, then Lemma \ref{PROPFirstloopLEM2} implies $x_0\geq c \sqrt{\beta_0}$. However, if $M=|y_0|^{\bar{m}}$, there are two cases to consider. If $y_0^m\geq 2 \beta_0$, then we have $x_0^2=\sigma\beta_0 + y_0^m\geq y_0^m/2=M^2/2$. If $y_0^m< 2 \beta_0$, since $\sigma\beta_0 + y_0^m\geq 0$ implies $|y_0|^m\leq 2 \beta_0$, Lemma \ref{PROPFirstloopLEM2}  gives $x_0\geq c \sqrt{\beta_0} \geq cM/\sqrt{2}$. In conclusion, by setting $c_1:=\min\{1,c\}/\sqrt{2}$ and $c_2:=\sqrt{2}$, and by recalling that $\beta_0\leq M^2$, we obtain
\begin{eqnarray}\label{EQ1732}
c_1 M \leq x_0 \leq c_2 M \quad \mbox{and} \quad \frac{\beta_0}{c_2 M}\leq  \frac{\beta_0}{x_0} \leq  \frac{\beta_0}{c_1M} \leq \frac{M}{c_1}.
\end{eqnarray}
Given a parameter $a\in (0,1]$ to be fixed later, we consider the interval $I_a:=[t_0-a\beta_0/x_0,t_0]$. Using the bounds $|\dot{\omega}_1|\leq 1$ and  (\ref{EQ1732}), we note that we have
\begin{eqnarray}\label{eq:bound_omega1beta}
 \left(c_1- \frac{a}{c_1} \right) M \leq x_0 - \frac{a\beta_0}{x_0} \leq \omega_1(t) \leq x_0 + \frac{a\beta_0}{x_0}\leq \left(c_2+\frac{a}{c_1}\right)M \quad \forall t\in I_a\cap[0,L(\omega)].
\end{eqnarray}
Thus, assuming $a \leq c_1^{2}$, we have $I_a\subset [0,L(\omega)]$, and,  using $|\dot\omega_2|\leq 1$ and (\ref{EQ1732}),
we obtain
\begin{eqnarray}\label{eq:bound_omega2beta}
|\omega_2(t)| \leq |y_0| +  \frac{a\beta_0}{x_0} \leq M^{\frac{2}{m}} +  \frac{a M}{c_1} \leq 2M^{\frac{2}{m}} \qquad \forall t \in I_a.
\end{eqnarray}
Using (\ref{eq:bound_omega1beta}) and (\ref{eq:bound_omega2beta}), we can bound the derivative of $P(t)=P(\omega(t))$ on $I_a$ as follows:
\begin{eqnarray}\label{dotP}
	|\dot P(t)| \leq 2\omega_1(t) +m |\omega_2(t)|^{m-1} \leq C_1 M  + C_2 M^{2-\frac{2}{m}} \leq C M \qquad \forall t \in I_a,
\end{eqnarray}
where $C_1:=2(c_1+c_2), C_2:=m2^{m-1}$ and $C=C_1+C_2$, provided $\epsilon>0$ is small enough. We now fix $a\in (0,c_1^2]$ such that $c_1-a/c_1 \geq 1/(2c_1)$ and $aC/c_1 \leq 1/2$. From (\ref{eq:bound_omega1beta}) and (\ref{dotP}), it follows that
\begin{eqnarray}\label{boundP}
\omega_1(t) \geq \frac{M}{2c_1} \quad \mbox{and} \quad |P(t)| \geq \beta_0 - \frac{a\beta_0}{x_0} C M \geq \beta_0 -  \frac{aC\beta_0}{c_1} \geq \frac{\beta_0}{2} \qquad \forall t \in I_a.
\end{eqnarray}
From Lemma \ref{LEMAll0} (ii), we have $\int_{\ell_0} |\dot{\theta}(t)|\, dt < 2\pi$. Moreover, Lemma \ref{LEMAll0} (viii) implies that $\int_{\tau}^{s_0^-} |\dot{\theta}(t)|\, dt \leq 2 \pi$, where $\tau=\tau_0$ if $J_0\subset I_0$ and $\tau=\tau_1$ if $J_0\subset I_1$. If $\tau=\tau_0$, we have
\[
\int_{I_a} |\dot\theta(t)| dt \leq \int_0^{t_0}|\dot\theta(t)| dt \leq  \int_{0}^{s_0^-} |\dot\theta(t)|  \, dt + \int_{s_0^-}^{s_0^+} |\dot\theta(t)|  \, dt \leq 4\pi.
\]
If $\tau=\tau_1$, the curve $\omega_{I_0}$ is injective. Using the Gauss--Bonnet formula (\ref{GaussBonnet}) applied to the positively oriented closed simple curve $\eta:[0,\tau_1+\tau'] \rightarrow \R^2$ formed by concatenating $\omega|_{ [0,\tau_1]}$ with the segment $P_{ [\omega(\tau_1),\omega(0)]}$ of $\{P=0\}$, joining $\omega(\tau_1)$ to $\omega(0)$, reparametrized by arc length on an interval $[0,\tau']$, we obtain
\[
2 \pi =  \int_{0}^{\tau_{1}} \kappa(t) \, dt + \int_{0}^{\tau'} \bar{\kappa}(t) \, dt  + \delta_0 + \delta_1,
\]
where $\kappa$ and $\bar{\kappa}$ are the signed curvature of $\omega$ and $P_{ [\omega(\tau_1),\omega(0)]}$, respectively, and $\delta_0, \delta_1 \in [-\pi,\pi]$ represent the discontinuities of curvature of $\eta$ at $t=0$ and $t=\tau_1$. Given that $\tau'\leq 1$ and $|\bar{\kappa}|\leq 1$ for sufficiently small $\epsilon>0$, recalling (\ref{GaussBonnet1}) gives
\[
\int_{I_a} |\dot\theta(t)| dt \leq  \int_0^{\tau}|\dot\theta(t)| dt +  \int_{\tau}^{t_0}|\dot\theta(t)| dt = \left| \int_0^{\tau}\kappa(t) dt\right| +  \int_{\tau}^{t_0}|\dot\theta(t)| dt \leq 8\pi +1.
\]
Finally, using (\ref{SYSomega}), (\ref{EQ1732}) and \eqref{boundP}, we conclude  that
\[
8\pi+1\geq \int_{I_a} 4|\lambda| \, \omega_1(t) \, |P(t)| \, dt \geq \frac{a\beta_0}{x_0} \left(4 |\lambda|  \frac{M }{2c_1} \frac{\beta_0}{2} \right)\geq  \frac{|\lambda| \beta_0^2}{c_1c_2},
\]
which completes the proof.
\end{proof}

We can now state the lemma required to end the proof of Proposition \ref{PROPfirstloop}.

\begin{lemma}\label{LEM34janvier}
There is $c_0>0$ such that $L_0\geq c\delta_0$.
\end{lemma}

\begin{proof}
If $L_0\geq x_0$, the result follows immediately from  (\ref{3janv1}), because  $L_0 \geq x_0  \geq c^2 \beta_0/x_0 \geq c^2 \delta_0$.
We now address the case where $L_0< x_0$. By Lemma \ref{LEMAll0} (ii), (\ref{SYSomega}), and $|J_0|=L_0$, we obtain
\begin{multline*}
\pi \leq \int_{J_0} \left|\dot{\theta}(t)\right| \, dt = |\lambda| \int_{J_0} \left|Q(\omega(t))\right| \, dt \leq |\lambda| L_0 \max_{t\in J_0} |Q(\omega(t))|\\
 \leq 4|\lambda| L_0 \beta_0 \max_{t\in J_0}\omega_1(t)) \leq 4|\lambda| L_0 \beta_0 \left(x_0+L_0\right).
\end{multline*}
Using $L_0<x_0$ and Lemma \ref{LEM3janvier}, it follows that
\[
L_0 \geq \frac{\pi}{4|\lambda| \beta_0 (x_0+L_0)} \geq \frac{\pi}{8|\lambda| \beta_0 x_0} \geq \frac{\pi \beta_0}{8Cx_0} \geq \frac{\pi \beta_0}{8C \max \{x_0,|y_0|^{\bar{m}}\} } = \frac{\pi}{8C} \, \delta_0,
\]
which completes the proof of the lemma.
\end{proof}

\begin{proof}[Proof of Proposition  \ref{PROPfirstloop}]
Lemmas \ref{LEM34janvier} and \ref{PROPFirstloopLEM1} imply that $y_0=\omega_2(t_0)\geq c(c_0)\epsilon$. From Lemma \ref{PROPPrelLEM1} (viii), we have $L_0  = o(\epsilon^{\bar{m}})$. Then $\omega_2(s_0^-) \geq c(c_0)\epsilon/2$ for $\epsilon>0$ small enough. The conclusion follows from Lemma \ref{PROPPrelLEM2}.
\end{proof}

\subsection{End of proof of Proposition \ref{PROPMain}}

Assertions (i) and (ii) have already been proven in Lemma \ref{PROPPrelLEM1}. To prove (iii) we argue by contradiction and assume that $\lambda >0$. We begin by applying several assertions from Lemma \ref{LEMAll0}. By (iii) and (vii), if $\omega|_{ I_0}$ admits a loop, it must be unique. Therefore, by (ix),  we have $\omega_2 (\tau_1)\leq \omega_2(0)=0$, which contradicts (i) (note that $P(\omega(\tau_1))=0$). Hence, $\omega_{I_0}$ is injective. From this, we conclude that the domain enclosed by $\omega|_{ I_0}$ and the curve $\{P=0, \,x_1\geq 0\}$ is convex. Hence the maximum of $P$ is attained on $\mbox{spt}(\omega|_{[0,\tau_1]})$. A simple computation shows that the maximum of $P$ over a segment $[0,(t^{\bar{m}},t)]$, for $t>0$, is given by $\alpha t^m$ for some constant $\alpha >0$. Hence, by Lemma \ref{PROPPrelLEM1} (v), we obtain
\[
\alpha \omega_2(\tau_1)^m\leq \max _{s\in [0,\tau_1]} P(\omega(s)) \leq \beta \leq C \, \epsilon^{3\bar{m}-1}.
\]
By Lemma \ref{LEMAll0} (iv) and (v), we have $1 \in \mathcal{I}^-$, and $\omega|_{ I_1}$ contains at least one loop. Furthermore, by Proposition \ref{PROPfirstloop} and Lemma \ref{LEMAll0} (vii), it contains exactly one loop, provided $\epsilon>0$ is sufficiently small, specifically $\epsilon<\epsilon(c)$, where $c$ is given by Proposition \ref{PROPfirstloop}. Therefore, by Lemma \ref{LEMAll0} (ix), we must have $\omega_2(\tau_1)\geq\omega_2(\tau_{2})$. Using the above inequality, we deduce
\[
\omega_2(\tau_{2}) \leq \omega_2(\tau_1)  \leq \left(\frac{C}{\alpha}\right)^{\frac{1}{m}} \epsilon^{\frac{3}{2}-\frac{1}{m}},
\]
which, for sufficiently small $\epsilon>0$, contradicts the lower bound $\omega_2(\tau_2)\geq c\epsilon$ provided by Proposition \ref{PROPfirstloop}.\\

Before we proceed with the proof of the remaining assertions, we note that by Proposition \ref{PROPfirstloop}, the upper bound on $L_0$ given by Lemma \ref{PROPPrelLEM1} (viii), the lower bound on $L_0$ provided by Lemma \ref{LEM34janvier}, and the inequality $\max\{x_0,|y_0|^{\bar{m}}\} \leq C\epsilon^{\bar{m}}$ (which follows by Lemma \ref{PROPPrelLEM1} (iv)), we may assume that
\begin{eqnarray}\label{EQ35janv1}
y_0 \geq c\epsilon \quad \mbox{and} \quad c\beta_0 \epsilon^{-\bar{m}} \leq L_0  \leq C\beta^{1-\frac{1}{m}},
\end{eqnarray}
for some constants $c,C>0$ and sufficiently small $\epsilon>0$. We further claim that we may also assume
\begin{eqnarray}\label{EQ35janv2}
|\lambda| \beta_0^2\geq c.
\end{eqnarray}
Consider a time interval $[t_1,t_2] \subset J_{0}$ where the times $t_1$ and $t_2$ will be chosen later. By (\ref{SYSomega}) and since $\omega_1$ and $\omega_2$ do not vanish on $J_{0}$ according to Proposition \ref{PROPfirstloop}, we have for any $t\in [t_1,t_2]$,
\begin{multline*}
  \frac{d}{dt} \left\{ 2\lambda P(\omega(t))^2 - 2 \sin \left(\theta(t)\right) - m \cos \left(\theta(t)\right) \frac{\omega_2(t)^{m-1}}{\omega_1(t)}\right\} = - m \cos\left(\theta(t)\right) \frac{d}{dt} \left\{ \frac{\omega_2(t)^{m-1}}{\omega_1(t)} \right\} \\
= - m \cos \left(\theta(t)\right) \frac{\omega_2(t)^{m-1}}{\omega_1(t)} \left( \cos \left(\theta(t)\right) + (m-1)  \sin \left(\theta(t)\right) \frac{\omega_1(t)}{\omega_2(t)}\right).
\end{multline*}
From Proposition \ref{PROPfirstloop} and Lemma \ref{PROPPrelLEM1} (iv), it follows that the absolute value of the right-hand side of the above equality is less than $C/\epsilon$ for some constant $C>0$ and sufficiently small $\epsilon>0$. Integrating this equality over $[t_1,t_2]\subset J_0$, we obtain
\begin{eqnarray}\label{eq:lP2-1}
\left|2\lambda \left(P(\omega(t_2))^2-P(\omega(t_1))^2\right)-2(\sin(\theta(t_2))-\sin(\theta(t_1)))\right| \leq |D|+\frac{CL_0}{\epsilon},
\end{eqnarray}
where
\begin{eqnarray}\label{eq:lP2-2}
D =m\left(\cos(\theta(t_2))\frac{\omega_2(t_2)^{m-1}}{\omega_1(t_2) }-\cos(\theta(t_1))\frac{\omega_2(t_1)^{m-1}}{\omega_1(t_1) }\right).
\end{eqnarray}
From Proposition \ref{PROPfirstloop} and Lemma \ref{PROPPrelLEM1} (iv), it follows  that $|D|\leq C\epsilon^{\bar{m}-1}$. Additionally,  from Lemma \ref{PROPPrelLEM1} (v) and (viii), we have $L_0 \epsilon^{-1} \leq C^{2-1/m}  \epsilon^{\bar{m}-1}$. Using (\ref{eq:lP2-1}), we deduce that by taking $C>0$ larger if necessary,
\[
\left| \lambda \left(P(\omega(t_2))^2-P(\omega(t_1))^2\right)\right|\geq |\sin(\theta(t_2))-\sin(\theta(t_1))|-C\epsilon^{\bar{m}-1}.
\]
By Lemma \ref{LEMAll0} (ii), the times $t_1$ and $t_2$ can be chosen so that $|\sin(\theta(t_2))-\sin(\theta(t_1))|\geq 1$. Therefore, by choosing $\epsilon_0>0$ small enough, we obtain
\[
\left|\lambda \left(P(\omega(t_2))^2-P(\omega(t_1))^2\right)\right|\geq 1/2.
\]
The inequality (\ref{EQ35janv2}) follows by noting that $|P(\omega(t_2))^2-P(\omega(t_1))^2|\leq 2\beta_0^2$. Before returning to the proof of the remaining assertions, we also need the following result.

\begin{lemma}\label{LEM32janvier}
Let $\bar{P}:=P\circ \omega : [0,L(\omega)] \rightarrow \R$ and $t_*\in [0,L(\omega)]$ such that $\bar{P}(t_*)\in (0,\beta_0)$, $\dot{\bar{P}}(t_*)=0$, $\ddot{\bar{P}}(t_*) \leq 0$, and $\sin \theta(t_*)>0$. Then $|\lambda| P(\omega(t_*))^{1+1/\bar{m}}\leq m(m-1)/8$.
\end{lemma}

\begin{proof}
Set $x_*=(x_*,y_*):=\omega(t^*)$, $\beta_*:=P(x_*)$ and $\theta_*:=\theta(t_*)$. We have
\begin{equation}\label{33janv1}
\left\{
 \begin{array}{l}
\dot{\bar{P}}(t_*) = 2x_*\cos \theta_* - m y_*^{m-1} \sin \theta_* = 0 \\
\ddot{\bar{P}}(t_*)  = 2\cos^2 \theta_* -2x_*\dot{\theta}(t_*) \sin \theta_* - m(m-1) y_*^{m-2} \sin^2 \theta_* - my_*^{m-1} \dot{\theta}(t_*)\cos \theta_*\leq 0,
 \end{array}
\right.
\end{equation}
with $\dot{\theta}(t_*) = 4\lambda x_* \beta_*$. Since $\beta_*=x_*^2-y_*^m>0$, we have $y_*\leq x_*^{2/m}$ (recall that $x_*>0$), so the inequality $\ddot{P}(t_*)\leq 0$ yields
\begin{multline}\label{33janv2}
4\lambda x_* \beta_* \left(-2x_*\sin \theta_* -my_*^{m-1}\cos \theta_*\right)  \leq  -2 \cos^2 \theta_* +m(m-1)y_*^{m-2}\sin^2 \theta_* \\
 \leq  m(m-1)y_*^{m-2} \leq m(m-1) x_*^{2-4/m},
\end{multline}
where, by the relation given by $\dot{\bar{P}}(t_*)=0$, the left-hand side satisfies
\begin{eqnarray}\label{33janv3}
4\lambda x_* \beta_* \left(-2x_*\sin \theta_* -my_*^{m-1}\cos \theta_*\right) = - 8\lambda x_*^2 \beta_* \sin \theta_* \left( 1+\cot^2 \theta_*\right) \geq 8\lambda x_*^2 \beta_*,
\end{eqnarray}
because $\lambda<0$ and $\sin (\theta_*)>0$. Then, combining (\ref{33janv2}) and (\ref{33janv3}), and using $x_*^2=\beta_*+y_*^m$, we infer that
\[
|\lambda| P(\omega(t^*))^{1+\frac{1}{\bar{m}}} = |\lambda| \beta_*^{1+\frac{1}{\bar{m}}} \leq |\lambda| \beta_* \left(\beta_*+y_*^m \right)^{\frac{2}{m}} = |\lambda| \beta_* x_*^{\frac{4}{m}} \leq \frac{m(m-1)}{8}.
\]
\end{proof}

Let us now prove assertions (iv) and (v). By (\ref{EQFirst1}), there is $\bar{i} \in \{0,1\}$ such that $J_0\subset I_{\bar{i}}$. Using Proposition \ref{PROPfirstloop}, Lemma \ref{PROPPrelLEM2}, and obstruction (viii) of Lemma \ref{LEMObs} we deduce that $\omega_{I_i}$ is injective, for all $i\neq \bar{i}$. As a consequence, if $J_0\subset I_0$ then, by Lemma \ref{LEMAll0} (iv) and obstruction (iv) of Lemma \ref{LEMObs}, we have $I_{\omega}=I_0$. Moreover, by Lemma  \ref{LEMAll0} (vii), $\ell_0$ is the unique loop of $\omega$. Next, we show that the case $J_0 \subset I_1$ cannot occur. In this case, there are two possibilities: either $I_{\omega}=I_0\cup I_1$ with $\omega|_{ I_0}$ injective, or $I_{\omega} = I_0\cup I_1 \cup I_2$ with both $\omega|_{ I_0}, \omega|_{ I_2}$ injective. Set $\beta_+:=\max_{t\in I_{\omega}}P$. For $i=0,2$ (or only $i=0$ in the first case), consider the simple closed curve $\eta$ defined as the concatenation of $\omega|_{ I_i}$ with the curve $\check{\omega}:[0,\tau:=\tau_{i+1}-\tau_i]\rightarrow \R^2$, corresponding to the segment of $\bar{\omega}$ connecting $\omega(\tau_{i+1})$ to $\omega(\tau_i)$. Since $\lambda<0$, Lemma \ref{LEMAll0} (ix) implies that $\omega(\tau_{i+1})\geq  \omega(\tau_i)$, and, by (\ref{SYSomega}) and (\ref{GaussBonnet1}), the signed curvature $\kappa$ of $\omega$ on $(\tau_i,\tau_{i+1})$ is negative. Consequently, the curve $\eta$ is positively oriented, $\omega(\tau_{i+1})\geq  \omega(\tau_i)$, the oriented angles  $\delta_i = \mbox{ang}(\dot{\omega}(\tau_{i+1}),\dot{\check{\omega}}(0))$ and $\delta_{i+1} = \mbox{ang}(\dot{\check{\omega}}(\tau),\dot{\omega}(\tau_i))$ lie in $(0,\pi)$, $\theta(\tau_{i+1}) < \theta(t) < \theta(\tau_i)$ for all $t\in (\tau_i,\tau_{i+1})$, and, by the Gauss--Bonnet formula (\ref{GaussBonnet}), we have
\[
2\pi -\delta_i - \delta_{i+1} = \int_{\tau_i}^{\tau_{i+1}} \kappa (t) \, dt + \int_{\tau_i}^{\tau_{i+1}} \check{\kappa} (t) \, dt,
\]
where $\check{\kappa}$ denotes the signed curvature of $\check{\omega}$. Since $\kappa\leq 0$, $\check{\kappa}$ tends to $0$ as $\epsilon\rightarrow 0$, and $\dot{\check{\omega}}(t)$ approaches the vertical vector $(0,-1)$ as $\epsilon\rightarrow 0$, we deduce that $\delta_i, \delta_{i+1}$ are close to $\pi$ for sufficiently small $\epsilon>0$, which implies that $\sin(\theta(t)) \geq 1/2$ for all $t\in I_i$. Therefore, by applying Lemma \ref{LEM32janvier} and (\ref{EQ35janv2}), we obtain a constant $D>0$ such that
\begin{eqnarray}\label{35janv4}
\beta_+ \leq D \beta_0^{\frac{2m}{m+2}}.
\end{eqnarray}
Next, by applying (\ref{19dec1}) from Proposition \ref{PROPsublevelsets} to the curve $\nu$ defined as the concatenation of $\omega|_{ [0,s_0^-]}$ with $\omega|_{[s_0^+,L(\omega)]}$, we obtain
\[
L(\omega) =  L(\nu) + L_0 \geq L(\bar{\omega}) - C \beta_+^{1-\frac{1}{m}} + L_0 \geq L(\bar{\omega}) - C D^{1-\frac{1}{m}}  \beta_0^{\frac{2(m-1)}{m+2}}+ L_0.
\]
Combining  this inequality with $L(\omega)\leq L(\bar{\omega})$ and the lower bound on $L_0$ obtained in (\ref{EQ35janv1}), we get
\[
\beta_0\epsilon^{-\bar{m}} \leq E \beta_0^{\frac{2m}{m+2}\left(1-\frac{1}{m}\right)} = E \beta_0^{\frac{2(m-1)}{m+2}} \leq E \beta_0,
\]
for some constant $E$ and $\epsilon>0$ small.  This leads to a contradiction, thereby concluding the proof that $I_{\omega}=I_0$. In summary, $\ell_0$ is the unique loop $\ell$ of $\omega$. It satisfies (iv) as a consequence of (\ref{EQ35janv1}), except for the last property ($\beta=  \max_{t\in J_{\ell}} | P(\omega(t))|$) (v) holds with $\beta_0$ instead of $\beta$ in the lower bound for $L(\ell)$, and (vi) is satisfied by (\ref{EQ35janv2}) and the relation $\beta\geq \beta_0$.  \\

To prove (vii), we consider a time $t_*\in (\tau_0,\tau_1) \setminus J_0$ where $P$ attains a local maximum (recall that $I_{\omega}=I_0=[\tau_0,\tau_1]$). By Lemma \ref{LEM32janvier}, the result follows if we show that $\sin(\theta(t_*))>0$. Suppose, by contradiction, that $\sin(\theta(t_*))\leq 0$. Then necessarily, by  \eqref{33janv1}, $\cos(\theta(t_*))<0$ and $\sin(\theta(t_*))\leq 0$, with $\sin (\theta(t_*))=0$ only if $\omega_2(t_*)=0$. Let $H_0$ denote the half-line $A_0 +[0,+\infty)(0,-1)$ and let $H$ denote the half-line $\omega(t_*)+ [0,+\infty) \dot{\omega}(t_*)$. Now, consider the planar curve $\Upsilon$ defined by
 \[
 \Upsilon := H_0 \cup \mbox{spt}(\omega|_{ [0,t_*]}) \cup H.
  \]
Let $U$ denote the connected component of $\R^2 \setminus \Upsilon$ that contains $\omega(t+\sigma)$ for small $\sigma>0$. Then, $A_{\epsilon}=\omega(\tau_1=L(\omega))\notin \bar{U}$. Therefore, the minimum $s_*\in (t_*,\tau_1]$ of $t\in (t_*,\tau_1)$ such that $\omega(t) \in \partial U$ is well-defined. We claim that $\omega|_{ [t_*,s_*]}$ has a loop contained in $U$. Indeed, by Lemma \ref{PROPPrelLEM1} (i), $\omega(s_*) \notin H_0$. Since $t_*\notin J_0$, we cannot have $\omega(t) \in \mbox{spt}(\omega|_{ [0,t_*)})$. Therefore, $\omega(s_*)\in H$, and by Lemma  \ref{LEMGaussBonnet}, $\omega|_{[t,s_*]}$ contains a loop that coincides with $\ell_0$, the unique loop of $\omega$. Furthermore, $\mbox{spt}(\ell_0)\subset U$. We now claim that
\[
P(\omega(t_*)) \geq \max \Bigl\{P(x) \, \vert \, x \in \bar{U}, \, x_2 \geq 0\Bigr\}.
\]
Since $\mbox{spt}(\ell_0)\subset U$, we have $\omega(s_0^-)\in U$, so the maximum of $x_2$ on $\bar{U}$ is greater than $c\epsilon$ (by (\ref{EQ35janv1})). Moreover, since $\sin \theta(t_*)\leq 0$, the maximum of $x_2$ on $\bar{U}$ is attained on $\mbox{spt}(\omega|_{[0,t_*]})$. Therefore, by Lemma \ref{PROPPrelLEM2} we infer that $\omega_2(t_*)>0$, for some $C>0$. In particular, from the formula for $\dot{\bar{P}}(t_*)=0$ given in (\ref{33janv1}), we conclude that $\cos\theta(t_*)<0$ and that $U$ is bounded and convex. Since the set of $x\in \R^2$ such that $P(x) >P(\omega(t_*))$, $x_2\geq 0$, and $x_1>0$ is convex and externally tangent to $U$ at $\omega(t_*)$, our claim is proved. If $\omega_1(s)\leq \omega_1(t)$ for all $s\in J_0$, then the claim implies that $Q(\omega(s))=4\omega_1(s) P(\omega(s))\leq Q(\omega(t_*))$ for all $s\in J_0$, which leads to a contradiction by obstruction (v) of Lemma \ref{LEMObs}. Therefore, there exists $\bar{s}\in J_0$ such that $\omega(\bar{s})\in U \cap \{x_1>\omega_1(t)\}$. Let $\mu>0$ be such that $\omega(\bar{s})-(0,\mu) \in \partial U$. Since $\cos(\theta(t_*)<0$, we have $x_1 < \omega_1(t_*)$ on $H \setminus \{\omega(t_*)\}$, which implies that $\omega(\bar{s}) -(0,\mu)\in \mbox{spt}(\omega|_{[0,t_*)})$, that is, $\mbox{spt} (\omega|_{[0,t_*)})+(0,\mu)\in \ell_0$. We conclude by obstruction (iii) of Lemma \ref{LEMObs}.\\

To complete the proof of (v), it remains to show that the lower bound for $L(\ell)$ holds with $\beta$ in place of $\beta_0$ and that the last property is satisfied. If the maximum of $P$ is attained outside $J_{\ell}=J_0$, then by (vii), $|\lambda| \beta^{1+1/\bar{m}}$ is bounded, which contradicts (\ref{EQ35janv2}) and (vi) for sufficiently small $\epsilon>0$. \\

To prove  (viii), we apply the Gauss--Bonnet formula to the concatenated curve $\eta$, which consists of $\omega|_{[0,s_0^-]}$, $\omega|_{[s_0^+,L(\omega]}$, the line segment $[A_{\epsilon},(0,\epsilon)]$, and the line segment $[(0,\epsilon),A_0]$. The integral of the signed curvature of $\eta$ is bounded by $3\pi+\pi/2\leq 4\pi $. Moreover, the total curvature of the loop is bounded by $2\pi$. We can thus conclude easily.

\appendix

\section{Proof of Proposition \ref{PROPsublevelsets}}\label{SECPROPsublevelsets}
Let $\epsilon, \rho>0$ be fixed. It follows from classical results of calculus of variations with constraints that the curve minimizing the length among all Lipschitz curves  $\zeta:[0,1]\rightarrow \R^2$ satisfying (\ref{CalcVar}) is the concatenated curve  $\nu_{\epsilon}^{\rho}$ (reparametrized on $[0,1]$) defined as follows (see Figure~\ref{fig1}):
\begin{eqnarray}\label{nu}
\nu_{\epsilon}^{\rho} := T_0 * \Gamma_{\rho}\left([t_0,t_1]\right) * T_1, \quad \mbox{with} \quad T_0 := \left[A_0,\Gamma_{\rho}(t_0) \right] \quad \mbox{and} \quad T_1 := \left[\Gamma_{\rho}(t_1), A_{\epsilon} \right].
\end{eqnarray}
Here, $t_0=t_0(\rho)$ is defined as the unique $t_0\geq 0$ for which the line segment $(A_0,\Gamma_{\rho}(t_0))$ is tangent to $\Gamma_{\rho}([0,+\infty))$ at $\Gamma_{\rho}(t_0)$ and $t_1=t_1(\rho,\epsilon) \geq 0$ is the unique $t_1 \geq 0$ such that the line segment $(A_{\epsilon}, \Gamma_{\rho}(t_1))$ is tangent to $\Gamma_{\rho}([0,+\infty))$ at $\Gamma_{\rho}(t_1)$. If the segment $[A_0,A_{\epsilon}]$ intersects $\Gamma_{\rho}([0,+\infty))$, then $t_0$ and $t_1$ are well-defined; otherwise we set $t_0=t_1:=0$.
\begin{figure}[H]
\begin{center}
\begin{tikzpicture}
\node[anchor=south west, inner sep=0] (image) at (0,0) {\includegraphics[width=6.8cm]{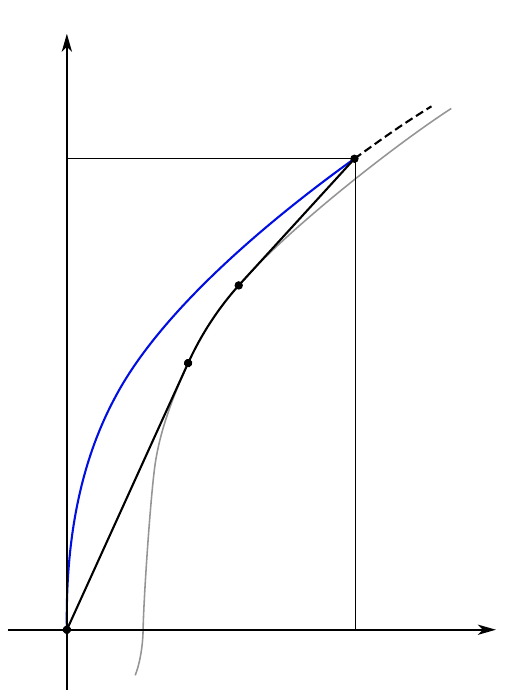}}; 

\begin{scope}[x={(image.south east)}, y={(image.north west)}]

    \node at (0.3, 0.55) {$\bar\omega$};

    \node at (0.45, 0.52) {$\nu_\epsilon^\rho$};

    \node at (0.57, 0.59) {$\Gamma_\rho(t_1)$};

    \node at (0.43, 0.44) {$\Gamma_\rho(t_0)$};

    \node at (0.7, 0.06) {$\epsilon^{m/2}$};

    \node at (0.66, 0.8) {$A_\epsilon$};

    \node at (0.09, 0.06) {$A_0$};

    \node at (0.32, 0.06) {$\rho^{1/2}$};

    \node at (0.1, 0.775) {$\epsilon$};

    \node at (0.12, 0.97) {$x_2$};

    \node at (0.97, 0.11) {$x_1$};

    \node at (0.92, 0.8) {$\{P=\rho\}$};

\end{scope}
\end{tikzpicture}
\caption{The curve $\nu_{\epsilon}^{\rho}$ in black\label{fig1}}
\end{center}
\end{figure}
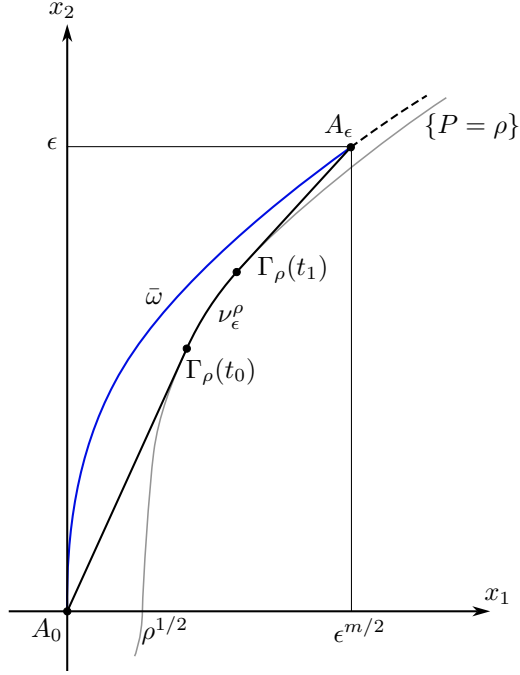
Hence proving \eqref{19dec1} is equivalent to prove
\begin{eqnarray}\label{19dec1_bis}
L\left(\nu_{\epsilon}^{\rho}\right)  \geq L\left(\bar{\omega}_{\epsilon}\right) - C(K) \rho^{1-\frac{1}{m}}.
\end{eqnarray}

Let us fix $K>0$ and assume that $\rho < K \epsilon^{3\bar{m}-1}$ (recall that $\bar{m}=m/2$).
The unique $t_0\geq 0$ such that the line $(A_0,\Gamma_{\rho}(t_0))$ is tangent to $\Gamma_{\rho}([0,+\infty))$  at $\Gamma_{\rho}(t_0)$ must satisfy $f_{\rho}(t_0)=t_0f_{\rho}'(t_0)$ and the unique $t_1 \geq 0$ such that the line $(A_{\epsilon}, \Gamma_z(t_1))$ is tangent to $\Gamma_{\rho}([0,+\infty))$  at $\Gamma_z(t_1)$ must satisfy $\epsilon^{\bar{m}}-f_{\rho}(t_1)=(\epsilon-t_1)f_{\rho}'(t_1)$ with $f_{\rho}'(t)= mt^{m-1}/(2f_{\rho}(t))$. From these equations, we deduce that
\begin{eqnarray}\label{LEM19decEQ1}
 \rho = \left( \bar{m} -1\right) t_0^m \quad \mbox{and} \quad \left( t_1^m+\rho\right)^{\frac{1}{2}}  = \epsilon^{\bar{m}} - \sigma_1\left( \epsilon -t_1\right) \quad \mbox{with} \quad \sigma_1=f_{\rho}'(t_1).
\end{eqnarray}
The first equation implies that $t_0=c\rho^{1/m}$ for some $c>0$. Using the assumption on $\rho$, it follows that $t_0=o(\epsilon)$ as $\epsilon \rightarrow 0$. From~\eqref{nu}, we have
\begin{eqnarray}\label{L0}
L(\nu_{\epsilon}^{\rho}) = L_1+L_2+L_3 \quad \mbox{with} \quad L_1:=L\left(T_0\right), \, L_2:= L\left(\Gamma_{\rho}([t_0,t_1])\right), \, L_3 := L\left(T_1\right).
\end{eqnarray}
We now proceed to derive lower bounds for $L_1, L_2$, and $L_3$. For $L_1$, we have
\begin{eqnarray}\label{L1}
L_1 = \left|(f_{\rho}(t_0), t_0)\right| =\sqrt{ t_0^2 + \left( t_0^m+\rho\right)} = \sqrt{ t_0^2 +\bar{m} t_0^m} = t_0 + \frac{\bar{m}}{2} t_0^{m-1} + o(t_0^{m-1})
\end{eqnarray}
by applying the Taylor expansion of $\sqrt{1+u}$ near $u=0$. For $L_2$, we note that, on the one hand,
\[
 L_2  = \int_{t_0}^{t_1} \sqrt{1+ \bar{m}^2 t^{2m-2} \left( t^m+\rho\right)^{-1}} \, dt \geq  \int_{t_0}^{t_1} \sqrt{1+ \bar{m}^2  t^{m-2} } \, dt = L\left( \bar{\omega}|_{ [t_0,t_1]}\right).
\]
On the other hand, we have
\[
L \left( \bar{\omega}|_{ [0,t_0]}\right) = t_0 + Ct_0^{m-1} + o (t_0^{m-1}),
\]
using (\ref{11dec0}), where $C=m^2/(8(m-1))$, and
\[
L\left( \bar{\omega}|_{ [t_1,\epsilon]}\right) = \int_{t_1}^{\epsilon} \sqrt{1+\bar{m}^2 t^{m-2}}\, dt \leq \left(\epsilon - t_1\right) \sqrt{1+\bar{m}^2 \epsilon^{m-2}}.
\]
Combining these, we find that
\begin{eqnarray}\label{L2}
L_2 \geq L\left( \bar{\omega}\right) -t_0 - Ct_0^{m-1} -\left(\epsilon - t_1\right) \sqrt{1+\bar{m}^2 \epsilon^{m-2}} +  o( t_0^{m-1}).
\end{eqnarray}
Finally, for $L_3$, using (\ref{LEM19decEQ1}), we have
\begin{eqnarray}\label{L3}
L_3 = \sqrt{ (\epsilon-t_1)^2+ \left( \epsilon^{\bar{m}}-(t_1^m+\rho)^{\frac{1}{2}}\right)^2} = \left(\epsilon -t_1\right) \sqrt{1+\sigma_1^2}.
\end{eqnarray}
By combining (\ref{L0})-(\ref{L3}) and defining $C':=\bar{m}/2-C$, we obtain
\[
L(\nu_{\epsilon}^{\rho}) \geq L\left( \bar{\omega}\right) +C' t_0^{m-1} + \left( \epsilon-t_1\right) \left[ \varphi \left(\sigma_1\right)-\varphi\left( \bar{m}\epsilon^{\bar{m}-1}\right)\right] + o ( t_0^{m-1}),
\]
where $\varphi: \R \rightarrow \R$ is defined as $\varphi(\lambda):=\sqrt{1+\lambda^2}$  for $\lambda\in \R$. Since by (\ref{LEM19decEQ1}) $t_0^{m-1}=c\rho^{1-1/m}$ for some $c>0$, proving \eqref{19dec1_bis} reduces to showing that the term  $\Delta:=( \epsilon-t_1) ( \varphi (\sigma_1)-\varphi ( \bar{m}\epsilon^{\bar{m}-1}))$ can be bounded from below by $-C\rho^{1-1/m}$ for some $C>0$, provided that $\epsilon>0$ is sufficiently small. Using the convexity of $\varphi$, we have
\begin{eqnarray}\label{19dec33}
\varphi \left(\sigma_1\right)-\varphi\left( \bar{m}\epsilon^{\bar{m}-1}\right) \geq \varphi' \left( \bar{m}\epsilon^{\bar{m}-1}\right) \left( \sigma_1- \bar{m}\epsilon^{\bar{m}-1}\right) \geq - \varphi' \left( \bar{m}\epsilon^{\bar{m}-1}\right) \left| \sigma_1- \bar{m}\epsilon^{\bar{m}-1}\right|,
\end{eqnarray}
where $\varphi'(\bar{m}\epsilon^{\bar{m}-1})=\bar{m}\epsilon^{\bar{m}-1}+o(\epsilon^{\bar{m}-1})$, and, by (\ref{LEM19decEQ1}), $\Delta_1:=\sigma_1- \bar{m}\epsilon^{\bar{m}-1}$ can be expressed as
\begin{eqnarray}\label{19dec34}
\Delta_1= \bar{m} \left[\frac{t_1^{m-1}}{\left(t_1^m+\rho\right)^{\frac{1}{2}}} -\epsilon^{\bar{m}-1}\right] = \bar{m} \epsilon^{\bar{m}-1}\left[ (1-\xi)^{\bar{m}-1} \left(1+ \frac{\alpha}{ (1-\xi)^{m}} \right) -1\right],
\end{eqnarray}
where we have introduced $\alpha=\rho/\epsilon^m$ and $\xi=1-t_1/\epsilon$. From (\ref{LEM19decEQ1}), it follows that $\alpha, \xi$ satisfy
\[
F(\alpha, \xi) =0 \quad \mbox{with} \quad F(\alpha, \xi) := 1- \left[ (1-\xi)^m+\alpha \right]^{\frac{1}{2}} - \bar{m} \xi (1-\xi)^{m-1}\left[(1-\xi)^m+\alpha\right]^{-\frac{1}{2}}.
\]
We note that $F(0,0)=0$, $\frac{\partial F}{\partial \alpha}(0,0) = -1/2$, and the only solution of $F(0,\xi)=0$ with $\xi \in [0,1]$ is $\xi=0$. By the implicit function theorem, this ensures the existence of $\delta_{\alpha}, \delta_{\xi}>0$ and a smooth function $\varphi: (-\delta_{\xi},\delta_{\xi}) \rightarrow \R$ such that, for any solution $(\alpha,\xi)$ of $F(\alpha,\xi)=0$ with $|\alpha|< \delta_{\alpha}$, we have $|\xi| < \delta_{\xi}$ and $\alpha=\varphi(\xi)$. Since $\rho<K\epsilon^{3\bar{m}-1}$ implies $\rho/\epsilon^m<\delta_{\alpha}$ for sufficiently small $\epsilon>0$, we conclude that $\alpha=\rho/\epsilon^m$ and $\xi=1-t_1/\epsilon$ satisfy $\alpha=\varphi(\xi)$ for small $\epsilon>0$, with $\xi$ tending to $0$ as $\epsilon\rightarrow 0$. Furthermore, since $\varphi'(0)=0$ and $\varphi''(0)\neq 0$, there exists $c>1$ such that $\xi^2/c \leq \alpha \leq c\xi^2$ for sufficiently small $\epsilon>0$. Consequently, by (\ref{19dec33})-(\ref{19dec34}), there is $c'>0$ such that
\[
\Delta  \geq  - \epsilon \xi \left( \bar{m}\epsilon^{\bar{m}-1}+o\left(\epsilon^{\bar{m}-1}\right)\right) \bar{m} \epsilon^{\bar{m}-1}\left( \left(\bar{m}-1\right) \xi + o(\xi)\right) \geq -c'\epsilon^{m-1}\xi^2,
\]
for $\epsilon$ sufficiently small. Finally, we note that $\epsilon^{m-1}\xi^2 \leq c\epsilon^{m-1}\alpha=c\rho\epsilon^{-1}\leq c\rho^{1-\frac{2}{3m-2}}\leq c\rho^{1-\frac{1}{m}}$ (recall that $m\geq 5$), completing the proof of \eqref{19dec1_bis}, and hence the one of (\ref{19dec1}).

It remains to prove (\ref{PROPsublevelsetslast}). The first and second  derivatives of   $f_{\rho}(\tau) = (\tau^m+\rho)^{\frac{1}{2}}$, defined for $\tau\geq 0$, are given by
\begin{eqnarray*}
f_{\rho}'(\tau) = \bar{m}\tau^{\bar{m}-1} \left(1+\frac{\rho}{\tau^m}\right) ^{-\frac{1}{2}} \mbox{ and }
f_{\rho}''(\tau)  = \bar{m} \tau ^{\bar{m}-2} \left(1 +\frac{\rho}{\tau^m} \right) ^{-\frac{1}{2}} \left[ (m-1)  - \bar{m} \left(1 +\frac{\rho}{\tau^m} \right) ^{-1}\right].
\end{eqnarray*}
These derivatives satisfy the bounds
\begin{eqnarray}\label{DP-DS}
0 \leq f'_{\rho}(\tau)  \leq \bar{m} \tau^{\bar{m}-1}\quad \mbox{and} \quad 0\leq  f''_{\rho}(\tau)  \leq \bar{m} \left(\bar{m}-1\right)\tau^{\bar{m}-2} \qquad \forall \tau \geq 0.
\end{eqnarray}
Since both $f_{\rho}$ and $\varphi(\lambda) =\sqrt{1+\lambda^2}$ are increasing on $[0,+\infty)$, the length of $\Gamma_{\rho}$ restricted to $[t,s]$, with $s\geq t\geq 0$, satisfies
\begin{eqnarray}\label{DP-DS2}
L(\Gamma_{\rho}|_{[t,s]})= \int_{t}^s \sqrt{1 +f_{\rho}'(\tau)^2}d\tau \leq (s-t) \sqrt{1 +f_{\rho}'(s)^2}.
\end{eqnarray}
Furthermore, by the convexity of $\varphi$ and $f_{\rho}$, we have
\begin{eqnarray*}
L([\Gamma_{\rho}(t),\Gamma_{\rho}(s)]) & = & (s-t) \varphi\left( \frac{f_{\rho}(s) - f_{\rho}(t)}{s-t}   \right) \nonumber\\
& \geq & (s-t) \left[ \varphi\left( f_{\rho}'(s)  \right) +\varphi' \left( f_{\rho}'(s)  \right)\left( \frac{f_{\rho}(s) - f_{\rho}(t)}{s-t}  -  f_{\rho}'(s) \right)  \right]\nonumber\\
& = & (s-t) \sqrt{1 +f_{\rho}'(s)^2} + \frac {(s-t) f'_{\rho}(s)}{\sqrt{1+ f'_{\rho}(s)^2 }} \left( f_{\rho}'(t)  -  f_{\rho}'(s) \right),
\end{eqnarray*}
which, by the mean value theorem, is equal to
\[
(s-t) \sqrt{1 +f_{\rho}'(s)^2} + \frac {(s-t)^2 f'_{\rho}(s)}{\sqrt{1+ f'_{\rho}(s)^2 }}  f_{\rho}''(u),
\]
for some $u\in [t,s]$. Then, (\ref{PROPsublevelsetslast}) follows by combining the inequality above with (\ref{DP-DS})-(\ref{DP-DS2}).

\end{document}